\documentclass[11pt]{amsart}

\usepackage{amssymb}
\usepackage{bm}
\usepackage[centertags]{amsmath}
\usepackage{amsfonts}
\usepackage{amsthm}
\usepackage{mathrsfs}

\theoremstyle{definition}

\newtheorem{thm}{Theorem}[section]
\newtheorem{dfn}[thm]{Definition}
\newtheorem{lem}[thm]{Lemma}
\newtheorem{prp}[thm]{Proposition}
\newtheorem{cor}[thm]{Corollary}
\newtheorem{rmk}[thm]{Remark}

\newtheorem*{thm*}{Theorem}

\newtheorem*{cor*}{Corollary}

\newtheorem*{prp*}{Proposition}

\newtheorem*{clm}{Claim}

\newtheorem*{ntt}{Notation}

\newcommand{\nn}{\mathbb{N}}
\newcommand{\inn}{\in\mathbb{N}}

\newcommand{\e}{\varepsilon}
\newcommand{\al}{\alpha}
\newcommand{\de}{\delta}
\newcommand{\adkx}{\al_k\big(\{x_i\}_i\big)}

\newcommand{\adnminusx}{\al_{n-1}\big(\{x_i\}_i\big)}
\newcommand{\alf}{$\alpha$-average}
\newcommand{\schr}{Schreier functional}
\newcommand{\la}{\lambda}
\newcommand{\La}{\Lambda}
\newcommand{\be}{\beta}
\newcommand{\X}{\mathfrak{X}_{_{^{0,1}}}^n}
\newcommand{\Xxi}{\mathfrak{X}_{_{^{0,1}}}^\xi}
\newcommand{\Xomega}{\mathfrak{X}_{_{^{0,1}}}^\omega}
\newcommand{\Sn}{\mathcal{S}_n}
\newcommand{\Sk}{\mathcal{S}_k}

\DeclareMathOperator{\supp}{supp}

\DeclareMathOperator{\ran}{ran}

\DeclareMathOperator{\scc}{succ}

\long\def\symbolfootnote[#1]#2{\begingroup%
\def\thefootnote{\fnsymbol{footnote}}\footnote[#1]{#2}\endgroup}

\begin{document}

\title [Strictly singular
operators in Tsirelson like spaces]{Strictly singular operators in
Tsirelson like spaces}

\author[S.A. Argyros]{Spiros A. Argyros}
\address{National Technical University of Athens, Faculty of Applied Sciences,
Department of Mathematics, Zografou Campus, 157 80, Athens,
Greece} \email{sargyros@math.ntua.gr}

\author[K. Beanland]{Kevin Beanland}
\address{Department of Mathematics and Applied Mathematics,
    Virginia Commonwealth University,
    Richmond, VA 23284}
\email{kbeanland@vcu.edu}

\author[P. Motakis]{Pavlos Motakis}
\address{National Technical University of Athens, Faculty of Applied Sciences,
Department of Mathematics, Zografou Campus, 157 80, Athens,
Greece} \email{pmotakis@central.ntua.gr}

\maketitle

\symbolfootnote[0]{\textit{2010 Mathematics Subject
Classification:} Primary 46B03, 46B06, 46B25, 46B45, 47A15}

\symbolfootnote[0]{\textit{Key words:} Spreading models, Strictly
singular operators, Invariant subspaces, Reflexive spaces}

\begin{abstract}
For each $n \inn$ a Banach space $\X$ is constructed is  
having the property that every normalized weakly null sequence generates either a $c_0$ or $\ell_1$ spreading models
and every subspace has weakly null sequences generating both $c_0$ and $\ell_1$ spreading models. The space $\X$ is also quasiminimal and 
for every infinite dimensional closed
subspace $Y$ of $\X$, for every $S_1,S_2,\ldots,S_{n+1}$ strictly
singular operators on $Y$, the operator $S_1S_2\cdots S_{n+1}$ is
compact. Moreover, for every subspace $Y$ as above, there exist
$S_1,S_2,\ldots,S_n$ strictly singular operators on $Y$, such that
the operator $S_1S_2\cdots S_n$ is non-compact.  
\end{abstract}








\section*{Introduction}

The strictly singular operators\footnote[1]{A bounded linear
operator is called strictly singular, if its restriction on any
infinite dimensional subspace is not an isomorphism.} form a two
sided ideal which includes the one of the compact
operators. In many cases, the two ideal coincide. This happens for
the spaces $\ell_p, 1\leqslant p <\infty, c_0$, as well as
Tsirelson space $T$ (see \cite{FJ}, \cite{T}). On the other hand,
in the spaces $L^p[0,1], 1\leqslant p < \infty, p\neq 2, C[0,1]$ the two
ideals are different. However, a classical result of V. Milman
\cite{M}, explains that in all the above spaces, the composition
of two strictly singular operators is a compact one. The aim of
the present paper, is to present examples of spaces where similar
properties occur in a hereditary manner. More precisely we prove
the following.

\begin{thm} For every $n\inn$ there exists a reflexive space with a
1-unconditional basis, denoted by $\X$, such that for every infinite
dimensional subspace $Y$ of $\X$ we have the following.

\begin{enumerate}

\item[(i)] The
ideal $\mathcal{S}(Y)$ of the strictly singular operators is
non-separable.

\item[(ii)] For every family
$\{S_i\}_{i=1}^{n+1}\subset \mathcal{S}(Y)$, the composition
$S_1S_2\cdots S_{n+1}$ is a compact operator.

\item[(iii)] There are
$S_1,\ldots,S_n \in S(Y)$, such that the
composition $S_1\cdots S_n$ is non-compact.

\end{enumerate}\label{thm0.1}

\end{thm}

The construction of the spaces $\X$ is based on T. Figiel's and W.B.
Johnson's construction of Tsirelson space \cite{FJ}, which is
actually the dual of Tsirelson's initial space \cite{T}. Therefore
the spaces $\X$ are Tsirelson like spaces and their norm is
defined through a saturation with constraints, described by the
following implicit formula, which uses the $n$\textsuperscript{th}
Schreier family $\Sn$.

For $x\in c_{00}$

\begin{equation*}
\|x\| = \max\big\{\|x\|_0,\;
\sup\{\sum_{q=1}^d\|E_qx\|_{j_q}\}\big\}
\end{equation*}
where the supremum is taken over all $\{E_q\}_{q=1}^d$ which are
$\Sn$-admissible successive finite subsets of $\mathbb{N}$,
$\{j_q\}_{q=1}^d$ very fast growing (i.e. $2\leqslant
j_1<\cdots<j_q$ and $j_q > \max E_{q-1}$, for $q>1$) natural
numbers and

\begin{equation*}
\|x\|_j = \sup\{\frac{1}{j}\sum_{q=1}^d\|E_qx\|\}
\end{equation*}
where the supremum is taken over all successive finite subsets of
the naturals $E_1<\cdots<E_d, d\leqslant j$.

Saturated norms under constraints were introduced by E. Odell and
Th. Schlumprecht \cite{OS1,OS2}. In particular the space defined in \cite{OS2}
has the property that every bimonotone basis is finitely block represented in
every subspace. Recently, in \cite{AM}, the first and third authors have used these techniques
to construct a reflexive hereditarily indecomposable space such that every
operator on an infinite dimensional subspace has a non-trivial invariant subspace.

Property (ii) of Theorem \ref{thm0.1}, combined with N. D. Hooker's and G. Sirotkin's
real version \cite{H},\cite{S} of V.I. Lomonosov's theorem \cite{L}, yields
that the strictly singular operators on the subspaces of $\X$
admit non trivial hyperinvariant subspaces.

Unlike the Tsirelson type spaces, the spaces $\X$ have
non-homogeneous asymptotic structure. In particular, every seminormalized
weakly null sequence admits either $\ell_1$ or $c_0$ as a
spreading model and every subspace $Y$ contains weakly null
sequences generating both $\ell_1$ and $c_0$ as spreading models.
As a result, the spaces $\X$ do not contain any asymptotic
$\ell_p$ subspace and, as a consequence, the spaces $\X$ do not contain
a boundedly distortable subspace \cite{MT}. The sequences in $\X$
generating $\ell_1$ spreading models admit a further classification
in terms of higher order $\ell_1$ spreading models. Recall that
for $k\inn$, a bounded sequence $\{x_i\}_{i\inn}$ generates an
$\ell_1^k$ spreading model if there exists $C>0$ such that
$\|\sum_{i\in F}\la_ix_i\| \geqslant C\sum_{i\in F}|\la_i|$ for
every $F\in\mathcal{S}_k$. The next proposition provides a precise description
of the possible spreading models of $\X$.

\begin{prp} Let $\{x_i\}_{i\inn}$ be a seminormalized weakly null
sequence in $\X$. Then one of the following holds.

\begin{itemize}

\item[(i)] $\{x_i\}_{i\inn}$ admits $c_0$ as a spreading model.

\item[(ii)] There exists $1\leqslant k \leqslant n$ such that
$\{x_i\}_{i\inn}$ admits an $\ell_1^k$ spreading model and it does
not admit  an $\ell_1^{k+1}$ one.

\end{itemize}

\label{0.2}
\end{prp}

The proof of Theorem \ref{thm0.1} (ii) is based on
Proposition \ref{0.2} and the following characterization of the non-strictly
singular operators on subspaces of $\X$.

\begin{prp} Let $Y$ be an infinite dimensional subspace of $\X$
and $T:Y\rightarrow Y$ a bounded linear operator. Then the
following are equivalent.

\begin{enumerate}

\item[(i)] The operator $T$ is not a strictly singular operator.

\item[(ii)] There exists $1\leqslant k \leqslant n$ and
a bounded weakly null sequence $\{x_i\}_{i\inn}$ such that both
$\{x_i\}_{i\inn}$ and $\{Tx_i\}_{i\inn}$ generate an $\ell_1^k$
spreading model and do not admit an $\ell_1^{k+1}$ one.

\item[(iii)] There exists $\{x_i\}_{i\inn}$ a bounded weakly null
sequence such that both $\{x_i\}_{i\inn}$ and $\{Tx_i\}_{i\inn}$
generate a $c_0$ spreading model.

\end{enumerate}\label{prop0.3}

\end{prp}

A space is called quasi-minimal if any two infinite dimensional subspaces have further
subspaces which are isomorphic. A major obstacle in proving the above, is to show
that certain normalized block sequences, that can be found in every subspace, are equivalent.
This also yields that the space $\X$ is
quasi-minimal. 

The above Proposition combined with the properties of the spreading models of the space $\X$ also allows us to study classes of strictly singular operators on subspaces of the spaces $\X$, which were introduced in \cite{ADST}. Recall that a bounded linear operator $T$ defined on a Banach space $X$, is said to be $\mathcal{S}_\xi$-strictly singular (the class is denoted $\mathcal{SS}_\xi(X)), for \xi<\omega_1$, if for every Schauder basic sequence $\{x_i\}_i$ in $X$ and $\e>0$, there exists a vector $x$ in the linear span of $\{x_i\}_{i\in F}$, where $F\in\mathcal{S}_\xi$ such that $\|Tx\| < \e\|x\|$. We prove that for $n\inn$ the space $\X$ satisfies the following: 
$$\mathcal{K}(Y)\subsetneq\mathcal{SS}_1(Y)\subsetneq\mathcal{SS}_2(Y)\subsetneq\cdots\subsetneq\mathcal{SS}_n(Y) = \mathcal{S}(Y)$$ 
and for every $1\leqslant k \leqslant n$, $\mathcal{SS}_k(Y)$ is a two sided ideal. This solves a problem in \cite{Te} by being the first example of a space for which the collection $\mathcal{SS}_k(\X)$ is a ideal not equal to $\mathcal{K}(\X)$ or $\mathcal{SS}(\X)$.

The spaces $\X$ can be extended to a transfinite hierarchy
$\mathfrak{X}_{_{^{0,1}}}^\xi$ for $1\leqslant \xi < \omega_1$.
Roughly speaking, the space $\Xxi$ is
defined with the use of the Schreier family $\mathcal{S}_\xi$ in
the place of $\Sn$. In section 5 we investigate the case the space $\Xomega$
and prove results analogous to those in the case of $\mathfrak{X}^1_{0,1}$.
We also comment, in passing, that for $\xi = \zeta
+ (n - 1)$ with $\zeta$ a limit ordinal satisfying $\eta + \zeta = \zeta$ for $\eta<\zeta$, the strictly singular
operators on the space $\Xxi$ behave in a similar manner as the spaces $\X$.

The paper is organized into six sections. The first one is devoted
to some preliminary concepts and results. In the second section we
introduce the norm of the space $\X$, by defining the norming set
$W$, a subset of $c_{00}$. The third section includes the
study of the spreading models generated by seminormalized
sequences of $\X$. Our approach uses tools similar to those in \cite{AM}. In particular, to each block
sequence $\{x_i\}_{i\inn}$ of $\X$, we assign a family of indices
$\adkx, k=0,\ldots,n-1$ and their behaviour determines the
spreading models generated by the subsequences of
$\{x_i\}_{i\inn}$. The fourth section contains the study of
equivalent block sequences in $\X$. The proof is rather involved
and based on the analysis of the elements of the set $W$. The
equivalence of block sequences is central to our approach and it is critical
in the proofs of Proposition \ref{prop0.3} which, in turn, proves
Theorem \ref{thm0.1}. The proofs of the latter results are given
in section five. In section six we provide the extended hierarchy
$\mathfrak{X}_{_{^{0,1}}}^\zeta, 1\leqslant \zeta < \omega_1$ and
we prove some of the fundamental properties of the spaces.

\section{Preliminaries}

\subsection*{The Schreier families} The Schreier families is
an increasing sequence of families of finite subsets of the
naturals, which first appeared in \cite{AA}, and is inductively defined in the
following manner.

Set $\mathcal{S}_0 = \big\{\{n\}: n\inn\big\}$ and $\mathcal{S}_1
= \{F\subset\mathbb{N}: \#F\leqslant\min F\}$.

Suppose that $\Sn$ has been defined and set $\mathcal{S}_{n+1} =
\{F\subset\mathbb{N}: F = \cup_{j = 1}^k F_j$, where $F_1 <\cdots<
F_k\in\Sn$ and $k\leqslant\min F_1\}$

If for $n,m\inn$ we set $\Sn*\mathcal{S}_m = \{F\subset\mathbb{N}:
F = \cup_{j = 1}^k F_j$, where $F_1 <\cdots< F_k\in\mathcal{S}_m$
and $\{\min F_j: j=1,\ldots,k\}\in\Sn\}$, then it is well known \cite{AD2} and follows easily by induction
that $\Sn*\mathcal{S}_m = \mathcal{S}_{n+m}$ .

\begin{dfn} Let $X$ be a Banach space, $\{x_i\}_{i\inn}$ be a
sequence in $X$, $k\inn$ and $1\leqslant p<\infty$. We say that
$\{x_i\}_{i\inn}$ generates an $\ell_p^k$ spreading model, if
there exists a uniform constant $C\geqslant 1$, such that for any
$F\in\mathcal{S}_k$, $\{x_i\}_{i\in F}$ is $C$-equivalent to the
usual basis of $(\mathbb{R}^{\#F},\|\cdot\|_p)$. The $c_0^k$ spreading models are defined
similarly.
\end{dfn}

\begin{rmk} Let $X, Y$ be Banach spaces and $T:X\rightarrow Y$ be
a bounded linear operator. If $\{x_m\}_{m\inn}$ is a bounded
sequence in $X$ such that $\{Tx_m\}_{m\inn}$ generates an
$\ell^k_1$ spreading model for some $k\inn$, then $\{x_m\}_{m\inn}$ generates an
$\ell^d_1$ spreading model, for some $d\geqslant k$.\label{remarklifting}
\end{rmk}

\begin{dfn} Let $X$ be a Banach space, $\{x_i\}_{i\inn}$ be a
seminormalized sequence in $X$ and $k\inn$. We say that
$\{x_i\}_{i\inn}$ generates a strong $\ell_1^k$ spreading model if
there exists  a seminormalized sequence $\{x_i^*\}_{i\inn}$ in
$X^*$ which generates a $c_0^k$ spreading model and $\e>0$, such
that $x_i^*(x_i) > \e$ for all $i\inn$ and $x^*_i(x_j) = 0$ for
$i\neq j$.
\end{dfn}

\begin{rmk}
If $X$ is a Banach space, $k\inn$, $\{x_i\}_i$ is a seminormalized weakly null sequence in $X$ generating a strong $\ell_1^k$ spreading model and $\{y_i\}_i$ is a sequence in $X$ with $\sum_{i=1}^\infty\|x_i-y_i\|<\infty$, then $\{y_i\}_i$ has a subsequence generating a strong $\ell_1^k$ spreading model.\label{rmkstrongell1knonblock}
\end{rmk}
The above is easily implied by the following.

\begin{lem}
Let $X$ be a Banach space, $\{x_i\}_{i\inn}$ be a
seminormalized weakly null sequence in $X$, $\{x_i^*\}_{i\inn}$ be a seminormalized w$^*$-null sequence in
$X^*$ and $\e>0$ such that $x_i^*(x_i) > \e$ for all $i\inn$ and $x^*_i(x_j) = 0$ for $i\neq j$. If $\{y_i\}_i$ is a sequence in $X$ with $\sum_{i=1}^\infty\|x_i-y_i\|<\infty$, then there exist a strictly increasing sequence of natural numbers $\{m_i\}_i$ and a seminormalized sequence $\{y_i^*\}_i$ in $X^*$ such that $y_i^*(y_{m_i}) > \e/2$ for all $i\inn$, $y^*_i(y_{m_j}) = 0$ for $i\neq j$ and $\sum_{i=1}^\infty\|y_i^* - x_{m_i}^*\| < \infty$.
\end{lem}

\begin{proof}
Using the fact that $\{x_i\}_i$ is weakly null, $\{x_i^*\}_i$ is w$^*$-null  and $\sum_{i=1}^\infty\|x_i-y_i\|<\infty$, we may pass to appropriate subsequences and relabel such that $\sum_{i\neq j}|x^*_i(y_j)| < \infty$. We may moreover assume that $\{y_i\}_i$ is Schauder basic and set $Y = [\{y_i\}_i]$. For $i\inn$, define a bounded linear functional $g_i:Y\rightarrow \mathbb{R}$ with $g_i(\sum_{j=1}^\infty c_jy_j) = \sum_{j\neq i}c_jx_i^*(y_j)$ and take $z_i^*$ to be a norm preserving extension of $g_i$ to $X$. Then the $y_i^* = x_i^* - z_i^*$ are the desired functionals.

\end{proof}

\begin{rmk} If a sequence generates a strong $\ell_1^k$ spreading model,
it generates an $\ell_1^k$ spreading model. Moreover, the class of
strong $\ell_1^k$ spreading models is strictly smaller than the
class of $\ell_1^k$ spreading models.
\end{rmk}

\subsection*{Special convex combinations} Next, we recall for $k \inn$ and $\e>0$
the notion of the $(k,\e)$ special convex combinations, (see
\cite{AGR,AT}). This is an important tool used throughout the paper.

\begin{dfn} Let $F\subset \nn$ and $x = \sum_{i\in F}c_ie_i$ be a vector in $c_{00}$.
Then $x$ is said to be a $(k,\e)$ basic special convex combination
(or a $(k,\e)$ basic s.c.c.) if:

\begin{enumerate}

\item[(i)] $F\in\mathcal{S}_k, c_i\geqslant 0$, for $i\in F$ and
$\sum_{i\in F}c_i = 1$.

\item[(ii)] For any $G\subset F, G\in\mathcal{S}_{k-1}$, we have
that $\sum_{i\in G}c_i < \e$.

\end{enumerate}

\end{dfn}

\begin{dfn}
Let $x_1 <\cdots<x_m$ be vectors in $c_{00}$ and $\psi(k) =
\min\supp x_k$, for $k=1,\ldots,m$. Then $x = \sum_{k=1}^mc_kx_k$
is said to be a $(n,\e)$ special convex combination (or $(n,\e)$
s.c.c.), if $\sum_{k=1}^mc_ke_{\psi(k)}$ is a $(n,\e)$ basic
s.c.c.\label{defscc}
\end{dfn}

\subsection*{Repeated averages}

For every $k\inn$ and $F$ a maximal $\mathcal{S}_k$ set we
inductively define the repeated average $x_F = \sum_{i\in
F}c_i^Fe_i$ of $F$, which is a convex combination of the usual
basis of $c_{00}$.

For $k=1$ and $F$ a maximal $\mathcal{S}_1$ set, we define $x_F =
\frac{1}{\#F}\sum_{i\in F}e_i$.

Let now $k>1$ and assume that for any $F$ maximal
$\mathcal{S}_{k-1}$ set the repeated average $x_F$ has been
defined. If $F$ is a maximal $\mathcal{S}_k$ set, then there exist
$F_1 <\cdots <F_d$ maximal $\mathcal{S}_{k-1}$ sets such that $F =
\cup_{q=1}^dF_q$. Set $x_F = \frac{1}{d}\sum_{q=1}^dx_{F_q}$.

The proof of the next proposition can be found in \cite[Chapter 2, Proposition 2.3]{AT}.

\begin{prp} Let $k\inn$ and $F$ be a maximal $\mathcal{S}_k$ set.
Then the repeated average of $F$ $x_F = \sum_{i\in F}c_ie_i$ is a
$(k,\frac{3}{\min F})$ basic s.c.c.
\end{prp}

The above proposition yields the following.

\begin{prp}
For any infinite subset $M$ of $\nn$, $k\inn$ and
$\e>0$, there exists $F\subset M, \{c_i\}_{i\in F}$, such that $x
= \sum_{i\in F}c_ie_i$ is a $(k,\e)$ basic s.c.c.
\end{prp}

\section{The space $\X$}

Let us fix a natural number $n$ throughout the rest of the paper. We start with the definition of the norm of the space $\X$.

\begin{ntt} Let $G\subset c_{00}$. If a vector $\al\in G$ is of
the form $\al = \frac{1}{\ell}\sum_{q=1}^df_q$, for some $f_1 <
\ldots < f_d \in G, d\leqslant \ell$ and $2 \leqslant \ell$, then
$\al$ will be called an $\al$-average of size $s(\al) = \ell$.

Let $k\inn$. A finite sequence $\{\al_q\}_{q=1}^d$ of
$\al$-averages in $G$ will be called $\mathcal{S}_k$ admissible if
$\al_1 <\ldots < \al_d$ and $\{\min\supp\al_q:
q=1,\ldots,d\}\in\mathcal{S}_k$.

A sequence $\{\al_q\}_q$ of $\al$-averages in $G$ will be called
{\em very fast growing} if $\al_1 < \al_2 <\ldots$, $s(\al_1) < s(\al_2)
<\cdots$ and $s(\al_q)
> \max\supp\al_{q-i}$ for $1<q$.

If a vector $g\in G$ is of the form $g = \sum_{q=1}^d \al_q$ for
an $\Sn$-admissible and very fast growing sequence
$\{\al_q\}_{q=1}^d\subset G$, then $g$ will be called a
\schr.
\end{ntt}

\subsection*{The norming set} Inductively construct a set
$W\subset c_{00}$ in the following manner.
Set $W_0 = \{\pm e_i\}_{i\inn}$. Suppose that $W_0,\ldots,W_m$ have been constructed. Define:
\begin{equation*}
W_{m+1}^\al = \big\{\al = \frac{1}{\ell}\sum_{q=1}^df_q:\quad f_1
< \ldots < f_d \in W_m, \ell\geqslant 2, \ell\geqslant d\big\}
\end{equation*}
\begin{equation*}
W_{m+1}^S\! =\! \big\{ g = \sum_{q=1}^d \al_q:\;
\{\al_q\}_{q=1}^d\subset W_m\; \Sn\text{-admissible and very fast
growing}\big\}
\end{equation*}
Define $W_{m+1} = W_{m+1}^\al\cup W_{m+1}^S\cup W_m$ and $W =
\cup_{m=0}^\infty W_m$.

For $x\in c_{00}$ define $\|x\| = \sup\{f(x): f\in W\}$ and $\X =
\overline{(c_{00}(\mathbb{N}),\|\cdot\|)}$. Evidently $\X$ has a
1-unconditional basis.

One may also describe the norm on $\X$ with an implicit formula.
For $j\inn, j\geqslant 2, x\in\X$, set $\|x\|_j =
\sup\{\frac{1}{j}\sum_{q=1}^d\|E_qx\|\}$, where the supremum is
taken over all successive finite subsets of the naturals
$E_1<\cdots<E_d, d\leqslant j$. Then by using standard arguments
it is easy to see that
\begin{equation*}
\|x\| = \max\big\{\|x\|_0,\;
\sup\{\sum_{q=1}^d\|E_qx\|_{j_q}\}\big\}
\end{equation*}
where the supremum is taken over all $\mathcal{S}_n$ admissible
finite subsets of the naturals $E_1<\cdots<E_k$, such that $j_q
> \max E_{q-1}$, for $q>1$.

\section{Spreading models of $\X$}
In this section the possible spreading models of block sequences
are determined. The method used for this, is based on the $\al_k$
indices of block sequences, which are defined below and are
similar to the corresponding one in \cite{AM}. We show that every subspace of $\X$ admits the same variety of
spreading models.

\subsection*{ Spreading models of block sequences in $\X$:}

\begin{dfn} Let $0\leqslant k\leqslant n - 1$, $\{x_i\}_{i\inn}$ be a
block sequence in $\X$ that satisfies the following. For any
subsequence $\{x_{i_j}\}_{j\inn}$ of $\{x_i\}_{i\inn}$, for any
very fast growing sequence of $\al$-averages $\{\al_q\}_{q\inn}$
and any $\{F_j\}_{j\inn}$ sequence of increasing subsets of the
naturals such that $\{\al_q\}_{q\in F_j}$ is $\mathcal{S}_k$
admissible we have that $\lim_j\sum_{q\in F_j}|\al_q(x_{i_j})| =
0$. Then we say that the $\al_k$-index of $\{x_i\}_{i\inn}$ is
zero and write $\adkx = 0$. Otherwise we write $\adkx > 0$.
\end{dfn}

The next proposition follow straight from the definition.

\begin{prp}
Let $0\leqslant k \leqslant n-1$ and $\{x_i\}_{i\inn}$ be a block
sequence in $\X$, then the following statements are equivalent.
\begin{itemize}

\item[(i)] $\adkx = 0$.

\item[(ii)] For any $\e>0$ there exist $j_0, i_0\inn$, such that
for any $\{\al_q\}_{q=1}^d$ very fast growing and
$\mathcal{S}_k$-admissible sequence of $\al$-average with
$s(\al_q)\geqslant j_0$ for $q=1,\ldots,d$ and for any $i\geqslant
i_0$, we have that $\sum_{q=1}^d|\al_q(x_i)| < \e$.

\end{itemize}\label{alternate}
\end{prp}

\begin{lem}
Let $\al$ be an $\al$-average in $W$, $\{x_k\}_{k=1}^m$ be a
normalized block sequence and $\{c_k\}_{k=1}^m$ non
negative reals with $\sum_{k=1}^mc_k = 1$. Then if $G_\al = \{k:
\ran\al\cap\ran x_k\neq\varnothing\}$, the following holds:
\begin{equation*}
|\al(\sum_{k=1}^m c_k x_k))| < \frac{1}{s(\al)}\sum_{i\in G_\al}c_i + 2\max\{c_i: i\in
G_\al\}.
\end{equation*}

\label{lem4}
\end{lem}

\begin{proof}
If $\al = \frac{1}{p}\sum_{j=1}^df_j$ with $d \leq p$. Set
\begin{eqnarray*}
E_1 &=& \{k\in G_\al:\;\text{there exists at most one}\;j\;\text{with}\;\ran f_j\cap\ran x_k\neq\varnothing\}\\
E_2 &=& \{1,\ldots,m\}\setminus E_1\\
J_k &=& \{j: \ran f_j\cap\ran
x_k\neq\varnothing\}\quad\text{for}\;k\in E_2.
\end{eqnarray*}
Then it is easy to see that
\begin{equation}
|\al(\sum_{k\in E_1}c_kx_k)| \leqslant \frac{1}{p}\sum_{k\in
G_\al}c_k. \label{lem4.6eq1}
\end{equation}
Moreover
\begin{equation}
|\al(\sum_{k\in E_2}c_kx_k)| < 2\max\{c_k: k\in
G_\al\}.\label{lem4.6eq2}
\end{equation}
Since $\#E_2 \leq 2p$ we have
\begin{equation*}
|\al(\sum_{k\in E_2}c_kx_k)| \leqslant \frac{1}{p}\sum_{k\in
E_2}c_k\big(\sum_{j\in J_k}|f_j(x_k)|\big) < \max\{c_k: k\in
G_\al\}\frac{2p}{p}
\end{equation*}

By summing up \eqref{lem4.6eq1} and \eqref{lem4.6eq2} the result
follows.

\end{proof}

\begin{lem}Let $1\leqslant k\leqslant n$, $x = \sum_{i=1}^mc_ix_i$
be a $(k,\e)$ s.c.c. with $\|x_i\|\leqslant 1$ for $i=1,\ldots,m$.
Let also $\{\al_q\}_{q=1}^d$ be a very fast growing and
$\mathcal{S}_{k-1}$-admissible sequence of $\al$-averages. Then
the following holds.
\begin{equation*}
\sum_{q=1}^d|\al_q(\sum_{i=1}^mc_ix_i)| < \frac{1}{s(\al_1)} + 6\e
\end{equation*}\label{lemqqq6}
\end{lem}

\begin{proof}

Set
\begin{eqnarray*}
G_1 &=& \{i:\;\text{there exists at most
one}\;q\;\text{with}\;\ran\al_q\cap\ran x_i\neq\varnothing\}\\
G_2 &=& \{i:\;\text{there exist at least
two}\;q\;\text{with}\;\ran\al_q\cap\ran x_i\neq\varnothing\}\\
J &=& \{q:\;\text{there exists}\;i\in
G_1\;\text{with}\;\ran\al_q\cap\ran x_i\neq\varnothing\}\\
G^q &=& \{i: \ran\al_q\cap\ran
x_i\neq\varnothing\}\quad\text{for}\;q\in J
\end{eqnarray*}
For $q\in J$, by Lemma \ref{lem4} it follows that
\begin{equation}
|\al_q(\sum_{i=1}^mc_ix_i)| < \frac{1}{s(\al_q)}\sum_{i\in G^q}c_i
+ 2\max\{c_i: i\in G^q\}\label{lemqqq6eq1}
\end{equation}
Choose $i_q\in G^q$ such that $c_{i_q} = \max\{c_i: i\in G^q\}$.
Since $\{\al_q\}_{q=1}^d$ is $\mathcal{S}_{k-1}$-admissible, it
follows that $\{\min\supp x_{i_q}: q\in J\}$ is the union of a
$\mathcal{S}_{k-1}$ set and a singleton. Therefore we conclude the
following.
\begin{equation}
\sum_{q\in J}\max\{c_i: i\in G^q\} < 2\e\label{lemqqq6eq2}
\end{equation}
Hence, combining \eqref{lemqqq6eq1} and \eqref{lemqqq6eq2}, we
have that
\begin{equation}
\sum_{q=1}^d|\al_q(\sum_{i\in G_1}c_ix_i)| < \frac{1}{s(\al_1)} +
4\e\label{lemqqq6eq3}
\end{equation}
Moreover, it is easy to see that $\{\min\supp x_i: i\in G_2\}$ is
the union of a $\mathcal{S}_{k-1}$ set and a singleton and
therefore we have the following.
\begin{equation}
\sum_{q=1}^d|\al_q(\sum_{i\in G_2}c_ix_i)|\leqslant \|\sum_{i\in
G_2}c_ix_i\| \leqslant \sum_{i\in G_2}c_i < 2\e\label{lemqqq6eq4}
\end{equation}
Finally, summing up \eqref{lemqqq6eq3} and \eqref{lemqqq6eq4}, the
desired result follows.

\end{proof}

\begin{prp} Let $0\leqslant k \leqslant n-1$, $\{x_i\}_{i\inn} \subset Ba(\X)$ be a normalized block
sequence.  The following hold:
\begin{enumerate}
\item[(i)] If $\adkx > 0$, then, by passing to a subsequence,
$\{x_i\}_{i\inn}$ generates a strong $\ell_1^{n-k}$ spreading
model. \item[(ii)] If $\al_{k^\prime}\big(\{x_i\}_i\big) = 0$ for
$k^\prime<k$ and $\{w_j\}_{j\inn}$ is a block sequence of
$\{x_i\}_{i\inn}$ such that $w_j = \sum_{i\in F_j}c_ix_i$ is a
$(n-k,\e_j)$ s.c.c. with $\lim_j\e_j = 0$, then
$\al_{n-1}\big(\{w_j\}_j\big) = 0$.
\end{enumerate}
\label{prop3}
\end{prp}

\begin{proof}

First we prove (i). Passing to a subsequence of $\{x_i\}_{i \inn}$
and relabeling we can find $\e>0$, a very fast growing sequence of
$\al$-averages $\{\al_q\}_{q \inn}$ and a sequence of successive
finite sets $(F_i)_{i=1}^\infty$  such that for $i \inn$ $\{
\al_q\}_{q \in F_i}$ is $\Sk$ admissible and
$$\sum_{q \in F_i} \al_q(x_i) \geqslant\e$$
for each $i \inn$.  Passing to a further subsequence and relabeling, we can assume that
$$\max \supp (\sum_{q \in F_i} \al_q ) < \min\supp x_{i+1}$$
for each $i \inn$. Set $x_i^* = \sum_{q \in F_i} \al_q$. Then
$x_i^*\in W$, $x^*_i(x_i) > \e$ for all $i\inn$ and $x^*_j(x_i) =
0$ for $i\neq j$. Therefore $\e<\|x_i^*\|\leqslant 1$ and all that
remains to be shown it that $\{x_i^*\}_{i\inn}$ generates a
$c_0^{n-k}$ spreading model.

  Let $F \in \mathcal{S}_{n-k}$. Note that $\{\al_q\}_{q \in \cup_{i\in F} F_i} $ is $\Sn$
admissible. It follows that $\|\sum_{i\in F} x_i^*\| \leqslant 1$. In other words, $\{x^*_i\}_{i \in\nn}$ generates a $c_0^{n-k}$
spreading model.

We now prove (ii).  Let  $w_j = \sum_{i \in F_j} c_i x_i$ be the
$(n-k,\e_j)$ s.c.c.; we claim that $\alpha_{n-1}(\{w_j\}_j) = 0$.
First, pass to a subsequence of $\{w_j\}$ and relabel for
simplicity. Now, fix a sequence $\{\alpha_{q}\}_{q \in
\mathbb{N}}$ of very fast growing $\alpha$-averages and a sequence
$(L_j)_{j \in \mathbb{N}}$ of successive finite subsets
$\mathbb{N}$ such that $\{\alpha_q\}_{q \in L_j}$ is
$\mathcal{S}_{n-1}$ admissible for each $j \inn$.

Let $\e>0$. First, we consider the case $k>0$. Since $\alpha_{k-1}(\{x_i\}_i)=0$ and $\{\alpha_{q}\}_{q \in \mathbb{N}}$ is very fast growing,
by Proposition \ref{alternate} we can find $q_0,i_0 \in \mathbb{N}$
such that for each finite set $L \geqslant q_0$, with $\{\alpha_q\}_{q \in L}$ being $\mathcal{S}_{k-1}$ admissible, and $i \geqslant i_0$, we have
$$ \sum_{q \in L} |\alpha_q (x_i)| < \varepsilon/3.$$
Find $j_0 \in \mathbb{N}$ such that for all $j \geqslant j_0$
\begin{equation}
\min L_{j} \geqslant q_0,~\min F_{j} \geqslant i_0 \mbox{ and } \varepsilon_{j} < \varepsilon/6 \label{e/6}
\end{equation}
Fix $j \geqslant j_0$. We claim that
$$\sum_{q \in L_j} |\alpha_{q}(w_j)| < \varepsilon.$$
This, of course, implies the $\al_{n-1}(\{w_j\}_j)=0$.
To simplify notation, let $L=L_j$ and $F=F_j$. Before passing to the proof we note the following:

For $i \in F$ and $E \subset L$ such that $\{\alpha_q\}_{q \in E}$
is $\mathcal{S}_{k-1}$
admissible, we have
\begin{equation}
\sum_{q \in E} |\alpha_q (x_i)| < \varepsilon /3. \label{e/3}
\end{equation}
Partition $L$ into the following sets:
\begin{equation*}
\begin{split}
& G_1= \{ i \in F : \mbox{ there is a unique $q \in L$ such that } \ran \alpha_q \cap \ran x_i \not= \emptyset \},\\
& G_2= \{ i \in F : \mbox{ there are at least two $q \in L$ such that } \ran \alpha_q \cap \ran x_i \not= \emptyset \}
\end{split}
\end{equation*}
First, consider the case of $G_1$.  For $q \in L$ let
$$H_q= \{ i \in G_1: \ran \alpha_q \cap \ran x_i \not= \emptyset \}.$$
If $q \not= q'$ then $H_q \cap H_{q'}= \emptyset$; and $\cup_{q \in L} H_q \subset F$.  Using (\ref{e/3}) (for singleton subsets of $L$) and the convexity of $(c_i)_{i \in F}$, we have
\begin{equation*}
\begin{split}
\sum_{q \in L} |\alpha_q ( \sum_{i \in G_1} c_i x_i)| &=  \sum_{q \in L} |\alpha_q ( \sum_{i \in H_q} c_i x_i)|\\
& <\frac{\varepsilon}{3}\sum_{q \in L} \sum_{i \in H_q} c_i \leqslant  \frac{\varepsilon}{3}.
\end{split}
\end{equation*}
For $i \in G_2$, set
\begin{equation*}
\begin{split}
& J_i = \{q \in L : \ran \alpha_q \cap \ran x_i \not=\emptyset\}\\
& G_2' =\{ i \in G_2 : \{ \min\supp \alpha_q : q \in J_i\} \not\in S_{k-1}\}.
\end{split}
\end{equation*}
This splits the estimates in the following way:
\begin{equation*}
\begin{split}
\sum_{q \in L} |\alpha_q ( \sum_{i \in G_2} c_i x_i )|&  \leqslant \sum_{i \in G_2} |c_i ( \sum_{q \in J_i} \alpha_q) (x_i)|\\
&= \sum_{i \in G'_2} c_i |( \sum_{q \in J_i} \alpha_q) (x_i)| + \sum_{i \in G_2 \setminus G'_2} c_i |( \sum_{q \in J_i} \alpha_q) (x_i)|
\end{split}
\end{equation*}
Since for each $i \in G_2 \setminus G'_2$, $\{\alpha_q\}_{i \in J_i}$ is $\mathcal{S}_{k-1}$ admissible we can apply (\ref{e/3}) to
conclude that
$$\sum_{i \in G_2 \setminus G'_2} c_i |( \sum_{q \in J_i} \alpha_q) (x_i)| \leqslant  \frac{\varepsilon}{3}
\sum_{i \in G_2 \setminus G'_2} c_i \leqslant
\frac{\varepsilon}{3}.
$$ For the final case, we must observe that
\begin{equation}
\{ \min\supp x_i :  i \in G_2' \} \in 2S_{n-k-1}. \label{2admiss}
\end{equation}
Let $G_2'' = G_2' \setminus \min G_2'$.  For each $i \in G_2''$ it is clear that
\begin{equation}
 \min \supp x_i \geqslant \min\supp \alpha_{\min J_{i'}} \mbox{ for } i' < i \mbox{ and }i' \in G_2'.\label{spread}
\end{equation}
Find $\ell \inn$ such that
$$\{ \min \supp \alpha_{\min J_i} : i \in G_2''\} \in S_\ell.$$
Since
$$ \{ \min\supp\alpha_q : q \in F \} \supset \bigcup_{i \in G_2''} \{ \min\supp \alpha_q : q \in J_i \}.$$
The second set is $\mathcal{S}_{n-1}$ admissible. It is clear that for $i \in G_2''$
$$ \min \supp \alpha_{\min J_i} = \min\{ \min\supp\alpha_q : q \in J_i \}$$
and $\{\min\supp \alpha_q : q \in J_i\} \in S_{d}$, for some $d\geqslant  k$.

The convolution property of the
Schreier sets yields that $\ell + d \leqslant  n-1$.  Therefore $\ell \leqslant  n-d-1 \leqslant  n-k-1$. From (\ref{spread}), it follows that
$$  \{ \min\supp x_i :   i \in G_2''  \}\in S_{n-k-1}.$$
Since we are excluding a singleton, $(\ref{2admiss})$ follows.
Therefore $\sum_{i \in G_2'} c_i < 2\varepsilon_j<\varepsilon/3$, by our choice of $j_0$ (see (\ref{e/6})). Since $\{x_i\}_i \subset Ba(\X)$
$$\sum_{i \in G'_2} c_i ( \sum_{q \in J_i} \alpha_q) (x_i) \leqslant  \sum_{i \in G'_2} c_i < \varepsilon/3.$$
This proves our claim for the case $k>0$.

Now we consider the case $k=0$.  Find $q_0 \in \nn$ such that
$$\frac{1}{s(\alpha_{q_0})}< \varepsilon/2.$$
Now fix $j_0 \inn$ such that for all $j \geqslant j_0$
$$\min L_j \geqslant q_0 \mbox{ and } \e_j < \e/8.$$
Fix $j \geqslant j_0$ and for simplicity let $L=L_j$ and $F =
F_j$.

Using Lemma \ref{lemqqq6} we have
\begin{equation*}
\sum_{q \in L} | \alpha_q (\sum_{i \in F} c_i x_i)| < \frac{\e}{2}
+ 4\cdot\frac{\e}{8} = \e
\end{equation*}
This finishes the proof.
\end{proof}

\begin{prp} Let $\{x_i\}_{i\inn} \subset Ba(\X)$ be a block
sequence such that $\adnminusx = 0$. Then for $\e>0$ there is a subsequence
$\{x'_i\}_{i \inn}$ of $\{x_i\}_{i \inn}$ such that for every $F \in \mathcal{S}_1$
$$\|\sum_{i \in F} x'_i \| < 1+\e.$$
Moreover if $\{x_i\}_{i \inn}$ is normalized there is a subsequence that
generates a spreading model isometric to $c_0$.
\label{cosm}
\end{prp}

\begin{proof}

Let $\{\e_i\}_{i\inn}$ be a summable sequence of positive reals,
such that $\e_i > 3\sum_{j>i}\e_j$ for all $i\inn$. Using
Proposition \ref{alternate} inductively choose a subsequence,
again denoted by $\{x_i\}_{i\inn}$, such that for any
$i_0\geqslant 2$ and $i\geqslant i_0$, for any
$\{\al_q\}_{q=1}^\ell$ very fast growing and
$\mathcal{S}_{n-1}$-admissible sequence of $\al$-average with
$s(\al_q)\geqslant\min\supp x_{i_0}$ for $q=1,\ldots,\ell$, we
have that
\begin{equation}
\sum_{q=1}^\ell|\al_q(x_i)| < \frac{\e_{i_0}}{i_0\max\supp x_{i_0
- 1}} \label{inductiveest}
\end{equation}

We will show that for any $t \leqslant  i_1 < \ldots < i_t$, $F \subset
\{ 1, \ldots t\}$ we have
$$|\al( \sum_{j\in F} x_{i_j})|<1+2\e_{i_{\min F}}$$
whenever $\al$ is an $\al$-average and
$$|g( \sum_{j \in F} x_{i_j})|< 1 + 3\e_{i_{\min F}}$$
whenever $g$ is a Schreier functional.  This implies the
conclusion of the proposition.

 For functionals in $W_0$ the above is clearly true. Assume, for some $m \geqslant 0$ the above
 holds for any $t \leqslant  i_1 < \ldots < i_t$ and any functional in $W_m$. In the first case, let
 $t \leqslant  i_1 < \ldots < i_t$ and $\al \in W_{m+1}$ with $\al= \frac{1}{\ell} \sum_{q=1}^d f_q, d \leqslant  \ell$.

Set
$$E_1 = \{ q : \mbox{ there exists at most one $j \leqslant  t$ such that } \ran f_q \cap \ran x_{i_j} \not= \emptyset \},$$
and $E_2 = \{1, \ldots , \ell\} \setminus E_1$.  For $q \in E_1$,
we have $|f_q(\sum_{j=1}^n x_{i_j})| \leqslant  1$. Therefore $\sum_{q
\in E_1} | f_q (\sum_{j=1}^n x_{i_j})| \leqslant  \# E_1$.

For  $q $ in $E_2$, let $j_q \in \{1, \ldots, t\}$ be minimum such
that $\ran x_{i_{j_q}} \cap \ran f_q \not= \emptyset$. If $q < q'$
are in $E_2$, $j_q < j_{q'}$.  By the inductive assumption

\begin{equation}
\begin{split}
\sum_{q \in E_2} | f_q ( \sum_{j=1}^t x_{i_j})| &< \sum_{q \in E_2} (1 + 3\e_{i_{j_q}}) \\
& < \# E_2 + 3\e_{i_1} + 3\sum_{j>1} \e_{i_j} < \# E_2 + 4
\e_{i_1}.
\end{split}
\end{equation}

\noindent Therefore

$$|\al( \sum_{j=1}^t x_{i_j})| < \frac{d +4 \e_{i_1}}{\ell} \leqslant 1+2\e_{i_1}.$$

Let $g\in W_{m+1}$ with  $g = \sum_{q=1}^d \al_q$ be a \schr. Set
\begin{eqnarray*}
j_0 &=& \min\{j: \ran g\cap\ran x_{i_j}\neq\varnothing\}\\
q_0 &=& \min\{q: \max \supp \al_q \geqslant \min\supp x_{i_{j_0 + 1}}\}
\end{eqnarray*}
Decompose $\{q: q>q_0\}$ into successive intervals
$\{J_\nu\}_{\nu=1}^{\nu_0}$ such that the following hold:
\begin{itemize}

\item[(i)] $\{q: q>q_0\} = \cup_{\nu=1}^{\nu_0}J_\nu$ and

\item[(ii)] $\{\min\supp\al_q: q\in J_\nu\}$ are maximal
$\mathcal{S}_{n-1}$ sets (except perhaps the last one).

\end{itemize}


Since $\{ \al_q\}_{q=1}^d$ is $\Sn$ admissible,
$\nu_0\leqslant\max\supp x_{i_{j_0}}$. By definition, for $q >
q_0$
$$s(\al_q) > \max \supp \al_{q_0} \geqslant \min \supp x_{i_{j_0 +
1}}.$$
Therefore we can apply $(\ref{inductiveest})$ to conclude that

\begin{equation}
\begin{split}
\sum_{q>q_0}|\al_q(\sum_{j=1}^tx_{i_j})| & =  \sum_{\nu=1}^{\nu_0} \sum_{q \in J_\nu} |\al_q(\sum_{j>j_0}^t x_{i_j})| .\\
&  <  \nu_0 \cdot \frac{\e_{i_{j_0 + 1}}}{i_{j_0 + 1}\max\supp x_{i_{j_0}}} \cdot t \\
& < \e_{i_{j_0}}.\label{otherest}
\end{split}
\end{equation}
For the other part of the functional, we consider two cases.

{\em Case 1:} Assume that for $q < q_0$, $\al_q  ( \sum_{j=1}^t
x_{i_j})=0$.  In this case we simply apply the inductive
assumption to conclude that $\al_{q_0}(\sum_{j=1}^t x_{i_j}) < 1 +
2\e_{i_{j_0}}$.  Combining this with (\ref{otherest}) finishes the
proof.

{\em Case 2:} If the first case does not hold we have that
$s(\al_{q_0}) \geqslant \min \supp x_{i_{j_0}}$.  Using
$(\ref{inductiveest})$ we have
\begin{equation}
\begin{split}
\sum_{q < q_0} |\al_q( \sum_{j=1}^t x_{i_j})| +  |\al_{q_0}( \sum_{j=1}^t x_{i_j})| & = \sum_{q<q_0} |\al_q( x_{i_{j_0}})| + |\al_{q_0}( \sum_{j=j_0}^t x_{i_j})|\\
&<  1 + \e_{i_{j_0}}.
\end{split}
\end{equation}

\noindent Combining this with (\ref{otherest}) gives the desired result.
\end{proof}

\begin{prp} Let $\{x_i\}_{i\inn}$ be a seminormalized block
sequence in $\X$ and $0\leqslant k \leqslant n-1$. The following
assertions are equivalent.
\begin{enumerate}

\item[(i)] $\al_{k^\prime}\big(\{x_i\}_i\big) = 0$ for
$k^\prime<k$.

\item[(ii)] $\{x_i\}_{i\inn}$ has no subsequence generating an
$\ell^{n-k+1}_1$ spreading model.

\end{enumerate}\label{propmmm7}
\end{prp}

\begin{proof}
First assume that (i) holds. Towards a contradiction, assume that
passing, if necessary, to a subsequence, $\{x_i\}_{i\inn}$
generates an $\ell^{n-k+1}_1$ spreading model, with a lower
constant $\theta
>0$.

We may choose $\{F_j\}_{j\inn}$ increasing $\mathcal{S}_{n-k}$
sets with $F_j\geqslant j$ for all $j\inn$, $\{\e_j\}_{j\inn}$
positive reals with $\lim_j\e_j = 0$ and $\{c_i\}_{\in F_j}$
positive reals, such that $w_j = \sum_{i\in F_j}c_ix_i$ is a
$(n-k,\e_j)$ s.c.c. for all $j\inn$.

If $M = \sup\{\|w_j\|: j\inn\}$, it follows that $\theta < \|w_j\|
\leqslant M$ for all $j\inn$.

For any $t\leqslant j_1 < \cdots < j_t$, $\cup_{q=1}^tF_{j_q}$ is
a $\mathcal{S}_{n-k+1}$ set, therefore

\begin{equation}
\|\sum_{q=1}^tw_{j_q}\| > \theta\cdot t\label{propmmm7eq1}
\end{equation}

Propositions \ref{prop3}(ii) and \ref{cosm}, yield that passing,
if necessary, to subsequence, for any $t\leqslant j_1 < \cdots <
j_t$ the following holds.

\begin{equation}
\|\sum_{q=1}^tw_{j_q}\| < 2M\label{propmmm7eq2}
\end{equation}

For $t$ appropriately large, \eqref{propmmm7eq1} and
\eqref{propmmm7eq2} together yield a contradiction.

Now assume that (ii) is holds. Let $0\leqslant k^\prime \leqslant
n-1$ such that $\al_{k^\prime}\big(\{x_i\}_i\big) > 0$.
Proposition \ref{prop3}(i) yields that passing, if necessary, to a
subsequence, $\{x_i\}_{i\inn}$ generates an $\ell_1^{n-k^\prime}$
spreading model. Since (ii) holds, we have that $n-k^\prime < n-k
+ 1$, therefore $k \leqslant k^\prime$ and this completes the
proof.

\end{proof}

\begin{prp} Let $\{x_i\}_{i\inn}$ be a seminormalized block
sequence in $\X$ and $0\leqslant k \leqslant n-1$. The following
assertions are equivalent.
\begin{enumerate}

\item[(i)] $\al_k\big(\{x_i\}_i\big) > 0$.

\item[(ii)] $\{x_i\}_{i\inn}$ has a subsequence generating a
strong $\ell^{n-k}_1$ spreading model.

\end{enumerate}\label{propmmm8}
\end{prp}

\begin{proof}
If (i) holds, then by Proposition \ref{prop3} so does (ii).

Assume now that (ii) is holds. Pass to a subsequence of
$\{x_i\}_{i\inn}$ generating an $\ell^{n-k}_1$ spreading model and
relabel for simplicity. Towards a contradiction assume that
$\al_k\big(\{x_i\}_i\big) = 0$.

Consider first the case $k = n-1$. Then by Proposition \ref{cosm},
$\{x_i\}_{i\inn}$ has a subsequence generating a $c_0$ spreading
model, which is absurd.

Otherwise, if $k < n-1$, then evidently we have that
$\al_{k^\prime}\big(\{x_i\}_i\big) = 0$ for $k^\prime < k+1$.
Proposition \ref{propmmm7} yields a contradiction.

\end{proof}

Combining Propositions \ref{cosm}, \ref{propmmm7} and
\ref{propmmm8}, we conclude the following.

\begin{cor} Let $\{x_i\}_{i\inn}$ be a normalized block sequence
in $\X$. Then the following assertions are equivalent.

\begin{enumerate}

\item[(i)] Any subsequence of $\{x_i\}_{i\inn}$ has a further
subsequence generating an isometric $c_0$ spreading model.

\item[(ii)] $\al_{n-1}\big(\{x_i\}_i\big) = 0$.

\end{enumerate}\label{cormmm9co}

\end{cor}

\begin{rmk} Every normalized weakly null sequence generating a $c_0$ spreading
model satisfies $\al_{n-1}\big(\{x_i\}_i\big) = 0$. The above
yields that $c_0$ spreading models generated by normalized weakly
null sequences are always isometric to the usual basis of
$c_0$.\label{isomcosm}
\end{rmk}

\begin{cor} Let $\{x_i\}_{i\inn}$ be a normalized block sequence
in $\X$ and $0\leqslant k\leqslant n-1$. Then the following
assertions are equivalent.

\begin{enumerate}

\item[(i)] $\al_k\big(\{x_i\}_i\big) > 0$ and
$\al_{k^\prime}\big(\{x_i\}_i\big) = 0$ for $k^\prime<k$.

\item[(ii)] $\{x_i\}_{i\inn}$ has a subsequence generating a
strong $\ell^{n-k}_1$ spreading model and no subsequence of
$\{x_i\}_{i\inn}$ generates a strong $\ell^{n-k+1}_1$ spreading
model.

\item[(iii)] $\{x_i\}_{i\inn}$ has a subsequence generating an
$\ell^{n-k}_1$ spreading model and no subsequence of
$\{x_i\}_{i\inn}$ generates an $\ell^{n-k+1}_1$ spreading model.

\end{enumerate}
\label{cormmm9ell1}
\end{cor}

\begin{proof}
Assume first that (i) holds. Propositions \ref{propmmm7} and
\ref{propmmm8} yield that (ii) also holds.

Assume now that (ii) is true. To prove that (iii) is true as well,
all that needs to be shown is that no subsequence of
$\{x_i\}_{i\inn}$ generates an $\ell^{n-k+1}_1$ spreading model.
Towards a contradiction, assume that this is not the case.
Proposition \ref{propmmm7} yields that there exists $k^\prime < k$
such that $\al_{k^\prime}\big(\{x_i\}_i\big) > 0$. In turn,
Proposition \ref{propmmm8} yields that $\{x_i\}_{i\inn}$ has a
subsequence that generates a strong $\ell^{n-k^\prime}_1$
spreading model. The fact that $k^\prime<k$ and no subsequence
$\{x_i\}_{i\inn}$ generates a strong $\ell^{n-k+1}_1$ spreading
model yields a contradiction.

For the last part, assume that (iii) holds. We will show that so
does (i). Proposition \ref{propmmm7} yields that
$\al_{k^\prime}\big(\{x_i\}_i\big) = 0$ for $k^\prime<k$. Towards
a contradiction, assume that $\al_k\big(\{x_i\}_i\big) = 0$.

If $k = n-1$, Corollary \ref{cormmm9co} yields that any
subsequence of $\{x_i\}_{i\inn}$ has a further subsequence
generating a $c_0$ spreading model, which is absurd.

Otherwise, if $k < n-1$, then $\al_{k^\prime}\big(\{x_i\}_i\big) =
0$ for $k^\prime<k + 1$. Once more, Proposition \ref{propmmm7}
yields that no subsequence of $\{x_i\}_{i\inn}$ generates an
$\ell^{n-k}_1$ spreading model, a contradiction which completes
the proof.

\end{proof}

Corollaries \ref{cormmm9co} and \ref{cormmm9ell1} easily yield the
following.

\begin{cor} Let $\{x_i\}_{i\inn}$ be a normalized weakly null
sequence in $\X$. Then passing, if necessary, to a subsequence,
exactly one of the following holds.

\begin{enumerate}

\item[(i)] $\{x_i\}_{i\inn}$ generates an isometric $c_0$
spreading model.

\item[(ii)] There exists $0\leqslant k\leqslant n-1$ such that
$\{x_i\}_{i\inn}$ generates a strong $\ell_1^{n-k}$ spreading
model and no subsequence of it generates an $\ell_1^{n-k+1}$
spreading model.

\end{enumerate}\label{cormmm9dich}

\end{cor}

\begin{rmk} Corollaries \ref{cormmm9ell1} and \ref{cormmm9dich} yield that
whenever a normalized weakly null sequence generates an
$\ell_1^{n-k}$ spreading model, for some $0\leqslant k\leqslant
n$, then passing, if necessary, to a subsequence, it generates a
strong $\ell_1^{n-k}$ spreading model.\label{strongell1}
\end{rmk}

As we will show in Proposition \ref{propmmm14}, any block subspace
of $\X$, hence any subspace of $\X$, contains a normalized weakly
null sequence generating a $c_0$ spreading model and for any
$0\leqslant k \leqslant n-1$, it contains a normalized weakly null
sequence generating an $\ell_1^{n-k}$ spreading model having no
subsequence generating an $\ell_1^{n-k+1}$ spreading model.

Although in the usual sense of spreading models, any subspace of
$\X$ admits exactly two types of them, in the sense of higher
order spreading models, any subspace of $\X$ admits exactly $n+1$
types.

It is an interesting question, whether for given $n\inn$ there
exists a Banach space $X$, such that any subspace of it admits
exactly $n+1$ types of spreading models, in the usual sense.

\subsection*{Spreading models of subspaces of $\X$}

\begin{prp}
Let $\{x_i\}_{i\inn}$ be a normalized block sequence in $\X$ that
generates a spreading model isometric to $c_0$, $\{F_j\}_{j\inn}$
be a sequence of successive subsets of the naturals, such that
$\#F_j\leqslant\min F_j$, for all $j\inn$ and $\lim_j\#F_j =
\infty$. Then if $y_j = \sum_{i\in F_j}x_i$, there exists a
subsequence of $\{y_j\}_{j\inn}$ generating an $\ell_1^n$
spreading model.\label{propmmm10}
\end{prp}

\begin{proof}
Since $\{x_i\}_{i\inn}$ generates a spreading model isometric to
$c_0$, it follows that $\|y_j\|\rightarrow 1$. By Proposition
\ref{prop3}, it suffices to choose $\{y_{j_m}\}_{m\inn}$ a
subsequence of $\{y_j\}_{j\inn}$, such that
$\al_0\big(\{y_{j_m}\}_m\big) > 0$. Set $j_1 = 1$ and assume that
$j_1,\ldots,j_{m-1}$ have been chosen. Set $d = \max\{\max\supp
y_{j_{m-1}},\#F_{j_{m-1}}\}$ and choose $j_m
> j_{m-1}$ such that $\#F_{j_m} > d$.

To see that $\{y_{j_m}\}_{m\inn}$ generates an $\ell_1^n$
spreading model, notice that for $m > 1$, there exists an
$\al$-average $\al_m$ with $\ran\al_m \subset \ran y_{j_m}$ and
$s(\al_m) = \#F_{j_m} > \max\{\max\supp \al_{m-1}, s(\al_{m-1})\}$
such that $\al_m(y_{j_m})\rightarrow 1$. Therefore
$\al_0\big(\{y_{j_m}\}_m\big) > 0$.

\end{proof}

\begin{cor} The space $\X$ does not contain seminormalized weakly null sequences generating $c_0^2$ or $\ell^{n+1}_1$ spreading
models.\label{cormmm11}
\end{cor}

\begin{proof}
Assume that there exists a seminormalized weakly null sequence
$\{x_i\}_{i\inn}$ generating a $c_0^2$ spreading model. We may
therefore assume that it is a block sequence. By Proposition
\ref{propmmm10}, it follows that there exist $\{F_j\}_{j\inn}$
increasing, Schreier admissible subsets of the naturals and
$\theta>0$, such that $\|\sum_{q=1}^m\sum_{i\in F_{j_q}}x_i\|
> \theta\cdot m$ for any $m\leqslant j_1 <\ldots< j_m$. Since for any such
$F_{j_1}<\cdots< F_{j_m}$ we have that $\cup_{q=1}^m
F_{j_q}\in\mathcal{S}_2$, it follows that $\{x_i\}_{i\inn}$ does
not generate a $c_0^2$ spreading model.

The fact that $\X$ does not contain seminormalized weakly null
sequences generating $\ell^{n+1}_1$ spreading models follows from
Corollary \ref{cormmm9dich}.

\end{proof}

\begin{prp} Let $0\leqslant k\leqslant n-1$ and $\{x_i\}_{i\inn}$
be normalized block sequence in $\X$ that generates an
$\ell^{n-k}_1$ spreading model and no subsequence of it generates
an $\ell^{n-k+1}_1$ spreading model. Then there exists
$\{F_j\}_{j\inn}$ and increasing sequence of subsets of the
naturals and $\{c_i\}_{i\in F_j}$ non-negative reals with
$\sum_{i\in F_j}c_i = 1$, satisfying the following. If we set $w_j
= \sum_{i\in F_j}c_ix_i$, then $\{w_j\}_{j\inn}$ is seminormalized
and generates a $c_0$ spreading model.\label{propmmm12}
\end{prp}

\begin{proof}
By Corollary \ref{cormmm9ell1} it follows that
$\al_k\big(\{x_i\}_i\big) > 0$ and
$\al_{k^\prime}\big(\{x_i\}_i\big) = 0$ for $k^\prime<k$. Choose
$\{F_j\}_{j\inn}$ and increasing sequence of $\mathcal{S}_{n-k}$
subsets of the naturals and $\{c_i\}_{i\in F_j}$ non negative
reals such that $w_j = \sum_{i\in F_j}c_ix_i$ is a $(n-k,\e_j)$
s.c.c. with $\lim_j\e_j = 0$.

Since $\{x_i\}_{i\inn}$ generates an $\ell^{n-k}_1$ spreading
model, it follows that $\{w_j\}_{j\inn}$ is seminormalized.
Moreover, Proposition \ref{prop3} (ii) yields that
$\al_{n-1}\big(\{w_j\}_j\big) = 0$. Applying Proposition
\ref{cosm} we conclude the desired result.

\end{proof}

\begin{prp} Let $\{x_i\}_{i\inn}$ be a normalized block sequence
in $\X$ generating an $\ell_1^n$ spreading model and $1\leqslant
k\leqslant n-1$. Then there exists $\{F_j\}_{j\inn}$ and
increasing sequence of subsets of the naturals and $\{c_i\}_{i\in
F_j}$ non-negative reals with $\sum_{i\in F_j}c_i = 1$, satisfying
the following. If we set $w_j = \sum_{i\in F_j}c_ix_i$, then
$\{w_j\}_{j\inn}$ is seminormalized, generates an $\ell^{n-k}_1$
spreading model and no subsequence of it generates an
$\ell^{n-k+1}_1$ spreading model. \label{propmmm13}
\end{prp}

\begin{proof}
By Corollary \ref{cormmm9ell1} and passing, if necessary to a
subsequence, there exists $\{\al_i\}_{i\inn}$ a very fast growing
sequence of $\al$-averages, such that $\ran \al_i\subset \ran x_i$
and $\theta > 0$ such that $\al_i(x_i) > \theta$ for all $i\inn$.
Choose $\{F_j\}_{j\inn}$ and increasing sequence of
$\mathcal{S}_k$ subsets of the naturals and $\{c_i\}_{i\in F_j}$
non-negative reals such that $w_j = \sum_{i\in F_j}c_ix_i$ is a
$(k,\e_j)$ s.c.c. with $\lim_j\e_j = 0$.

Since $\{x_i\}_{i\inn}$ generates an $\ell_1^n$ spreading model
and $k<n$, we have that $\{w_j\}_{j\inn}$ is seminormalized.

To see that $\{w_j\}_{j\inn}$ has a subsequence generating an
$\ell_1^{n-k}$ spreading model, by Corollary \ref{cormmm9ell1} it
is enough to show that $\al_{k}\big(\{w_j\}_j\big) > 0$. It is
straightforward to check that the sequences $\{\al_i\}_{i\inn}$
and $\{F_j\}_{j\inn}$ previously chosen, witness this fact.

It remains to be shown that no subsequence of $\{w_j\}_{j\inn}$
generates an $\ell^{n-k+1}_1$ spreading model. Once more, by
Corollary \ref{cormmm9ell1} it is enough to check that
$\al_{k-1}\big(\{x_i\}_i\big) = 0$.

Pass to a subsequence of $\{w_j\}_{j\inn}$, relabel for
simplicity, let $\{\al_i^\prime\}_{i\inn}$ be a very fast growing
sequence of $\al$-averages and $\{G_j\}_{j\inn}$ be an increasing
sequence of subsets of the naturals such that
$\{\al_i^\prime\}_{i\in G_j}$ is $\mathcal{S}_{k-1}$ admissible
for all $j\inn$. Lemma \ref{lemqqq6} yields the following.

\begin{equation*}
\lim_{j\to\infty}\sum_{i\in G_j}|\al_i^\prime(w_j)| \leqslant
\lim_{j\to\infty}\big(\frac{1}{s(\al_{\min G_j}^\prime)} +
6\e_j\big) = 0
\end{equation*}
By definition, this means that $\al_{k-1}\big(\{x_i\}_i\big) = 0$
and this completes the proof.

\end{proof}

\begin{prp} Let $Y$ be an infinite dimensional closed subspace of
$\X$. Then there exists a normalized weakly null sequence in $Y$
generating an isometric $c_0$ spreading model. Moreover for
$0\leqslant k\leqslant n-1$ there exists a sequence in $Y$ that
generates an $\ell_1^{n-k}$ spreading model and no subsequence of
it generates an $\ell_1^{n-k+1}$ one.\label{propmmm14}
\end{prp}

\begin{proof}
Assume first that $Y$ is a block subspace. We first show that $Y$
admits an isometric $c_0$ spreading model. Let $\{x_i\}_{i\inn}$
be a normalized block sequence in $Y$. If $\{x_i\}_{i\inn}$ has a
subsequence generating a $c_0$ spreading model, then by Remark
\ref{isomcosm} there is nothing to prove.

If this is not the case, by Corollary \ref{cormmm9co} we conclude
that $\al_{n-1}\big(\{x_i\}_i\big) > 0$. Set $k_0 =
\min\{k^\prime: \al_{k^\prime}\big(\{x_i\}_i\big) > 0\}$.
Corollary \ref{cormmm9ell1} yields that passing, if necessary, to
a subsequence, $\{x_i\}_{i\inn}$ generates an $\ell^{n-k_0}_1$
spreading model and no further subsequence of it generates an
$\ell^{n-k_0+1}_1$ one. Proposition \ref{propmmm12} yields that
$\{x_i\}_{i\inn}$ has a further seminormalized block sequence
$\{w_j\}_{j\inn}$ generating a $c_0$ spreading model. If we set
$y_j = \frac{w_j}{\|w_j\|}$, then by remark \ref{isomcosm}
$\{y_j\}_{j\inn}$ is the desired sequence.

We now prove that $Y$ admits an $\ell^n_1$ spreading model. Take
$\{x_i\}_{i\inn}$ a normalized block sequence in $Y$ generating an
isometric $c_0$ spreading model. By Proposition \ref{propmmm10}
there exists $\{w_j\}_{j\inn}$ a further block sequence of
$\{x_i\}_{i\inn}$ generating an $\ell^n_1$ spreading model. By
Corollary \ref{cormmm11} $\{w_j\}_{j\inn}$ is the desired
sequence.

Let $1\leqslant k\leqslant n-1$. We show that there exists a
sequence in $Y$ that generates an $\ell_1^{n-k}$ spreading model
and no subsequence of it generates an $\ell_1^{n-k+1}$ one. Let
$\{x_i\}_{i\inn}$ be a sequence in $Y$ generating an $\ell^n_1$
spreading model. Simply apply Proposition \ref{propmmm13} to find
the desired sequence.

Therefore the statement is true for block subspaces. The fact that
any subspace of $\X$ contains a sequence arbitrarily close to a
block sequence completes the proof.

\end{proof}

From this it follows that $\X$ cannot contain $c_0$ or $\ell_1$,
therefore from James' Theorem for spaces with an unconditional
basis \cite{J}, the next result follows.

\begin{cor} The space $\X$ is reflexive.\label{cormmm15}
\end{cor}

\begin{cor} Let $Y$ be an infinite dimensional, closed  subspace of $\X$. Then $Y^*$ admits a
spreading model isometric to $\ell_1$. Moreover, for $0\leqslant
k\leqslant n-1$ there exists a sequence in $Y^*$ generating a
$c_0^{n-k}$ spreading model, such that no subsequence of it
generates a $c_0^{n-k+1}$ one.\label{cormmm16}
\end{cor}

\begin{proof} Since $Y$ contains a sequence $\{x_i\}_{i\inn}$
generating a spreading model isometric to $c_0$, which we may
assume is unconditional Schauder basic, such that
$\{x_i\}_{i\geqslant j}$ has an unconditional basic constant
$c_j\rightarrow 1$, as $j\rightarrow\infty$, then for any
normalized $\{x_i^*\}_{i\inn}\subset Y^*$, such that $x^*_i(x_i) =
1$, we have that $\{x_i^*\}_{i\inn}$ generates a spreading model
isometric to $\ell_1$.

Let now $0\leqslant k\leqslant n-1$. Use Proposition
\ref{propmmm14} to choose $\{x_i\}_{i\inn}$ a normalized weakly
null sequence in $Y$, generating an $\ell_1^{n-k}$ spreading model,
such that no subsequence of it generates an $\ell_1^{n-k+1}$ one.

By Remark \ref{strongell1} and passing if necessary to a
subsequence, there exist $\e>0$ and $\{x_i^*\}_{i\inn}$ a
seminormalized sequence in $X^*$ generating a $c_0^{n-k}$
spreading model satisfying the following. $x_i^*(x_i) > \e$ for
all $i\inn$ and $\sum_{i\neq j}|x_i^*(x_j)| < \infty$. Since
$\{x_i\}_{i\inn}$ has no subsequence generating an
$\ell_1^{n-k+1}$ spreading model, it follows that
$\{x_i^*\}_{i\inn}$ has no subsequence generating a $c_0^{n-k+1}$
one.

Let $I^*:{\mathfrak{X}_{_{0,1}}^{n*}}\rightarrow Y^*$ be the dual
operator of $I:Y\rightarrow\X$. Then, since $\|I^*\| = 1$, to see
that this generates a $c_0^{n-k}$ spreading model, all that needs to
be shown is that $\{I^*x_i^*\}_{i\inn}$ is bounded from below.
Indeed, $\|I^*x_i^*\|\geqslant (I^*x_i^*)(x_i) = x_i^*(x_i) > \e$.

It remains to be shown $\{I^*x_i^*\}_{i\inn}$ has no subsequence
generating a $c_0^{n-k+1}$ spreading model. Since it is
seminormalized, $(I^*x_i^*)(x_i) = x_i^*(x_i) > \e$ for all
$i\inn$, and $\sum_{i\neq j}|(I^*x_i^*)(x_j)| = \sum_{i\neq
j}|x_i^*(x_j)| < \infty$ and $\{x_i\}_{i\inn}$ has no subsequence
generating an $\ell_1^{n-k+1}$ spreading model, the result easily
follows.

\end{proof}

\section{Equivalent block sequences in $\X$}

In this section we prove that the space $\X$ is quasi minimal by
showing that every two block subspaces have further block
sequences which are equivalent. Our method is based on the
analysis of the functionals of the norming set $W$ and we use some
techniques first appeared in \cite{AD}.\vskip3pt

In Tsirelson space, whenever two seminormalized block sequences
$\{x_m\}_{m\inn}$, $\{y_m\}_{m\inn}$ satisfy $x_m < y_{m+1}$ and
$y_m < x_{m+1}$ for all $m\inn$, then they are equivalent (see
\cite{CS}). In the space $\X$ this is false, since seminormalized
sequences satisfying this condition may be constructed generating
different spreading models, therefore they cannot be equivalent.

Even in the case for sequences satisfying the above condition,
which moreover generate the same spreading model, we are unable to
prove that they have equivalent subsequences, not even if they
only consist of elements of the basis. The reason for this is the
fact that when constructing Schreier functionals in the norming
set $W$, unlike the norming set of Tsirelson space, very fast
growing sequences of $\al$-averages need to be taken.

In order to compensate for this fact, the following is done. Let
$\{x_m\}_{m\inn}$, $\{y_m\}_{m\inn}$ be normalized block
sequences, both generating $\ell_1^n$ spreading models, such that
$x_m < y_{m+1}$ and $y_m < x_{m+1}$ for all $m\inn$. we show that
by appropriately blocking both sequences in the same manner, we
obtain sequences which are equivalent. More precisely, we prove
the following.

\begin{prp}
Let $\{x_m\}_{m\inn}, \{y_m\}_{m\inn}$ be normalized block
sequences in $\X$, both generating $\ell_1^n$ spreading models,
such that $x_m < y_{m+1}$ and $y_m < x_{m+1}$ for all $m\inn$.
Then there exist $\{F_m\}_{m\inn}$ successive subsets of the
naturals and $\{c_i\}_{i\in F_m}$ non-negative reals, for all
$m\inn$, such that if $z_m = \sum_{i\in F_m}c_ix_i$ and $w_m =
\sum_{i\in F_m}c_iy_i$, then $\{z_m\}_{m\inn}$ and
$\{w_m\}_{m\inn}$ are seminormalized and equivalent. \label{cor11}
\end{prp}

Our method for showing the equivalence of $\{z_m\}_{m\inn}$ and
$\{w_m\}_{m\inn}$ is based on the following. For every $f$ in $W$
there exist $g^1, g^2, g^3$ in $W$ such that $\theta f(z_m) <
g^1(w_m) + g^2(w_m) + g^3(w_m) + \e_m$, for some fixed constant
$\theta$ and $\{\e_m\}_{m\inn}$ a summable sequence of positive
reals. The choice of the $g^i$ uses the tree analysis of $f$ given
below. Clearly the roles of $\{z_m\}_{m\inn}$ and
$\{w_m\}_{m\inn}$ can be reversed and this yields the equivalence
of the two sequences.

\subsection*{The tree analysis of a functional $\mathbf{f\in W}$} Let $f\in
W$. We construct a finite, single rooted tree $\La$ and choose
$\{f_\la\}_{\la\in\La}\subset W$, which will be called a tree
analysis of $f$.

Set $f_\varnothing =f$, where $\varnothing$ denotes the root of
the tree to be constructed. Choose $m\inn$, such that $f\in W_m$.
If $m = 0$, then the tree analysis of $f$ is $\{f_\varnothing\}$.
Otherwise, if $f$ is a \schr, $f = \sum_{j=1}^df_j$, where
$\{f_j\}_{j=1}^d\subset W_{m-1}$ is a very fast growing and
$\Sn$-admissible sequence of $\al$-averages, set $\{f_j\}_{j=1}^d$
to be the immediate successors of $f_\varnothing$. If $f$ is an
\alf, $f = \frac{1}{n}\sum_{j=1}^df_j$, where $\{f_1<\cdots<
f_d\}\subset W_{m-1}$, set $\{f_j\}_{j=1}^d$ to be the immediate
successors of $f_\varnothing$.

Suppose that the nodes of the tree and the corresponding
functionals have been chosen up to a height $\ell < m$ such that
$f_\la\in W_{m-h(\la)}$. Let $\la$ be such that $h(\la)=\ell$. If
$f_\la\in W_0$, then don't extend any further and $\la$ is a
terminal node of the tree. If $f_\la$ is a \schr, $f_\la =
\sum_{j=1}^df_j$, where $\{f_j\}_{j=1}^d\subset W_{m-\ell-1}$ is a
very fast growing and $\Sn$-admissible sequence of $\al$-averages,
set $\{f_j\}_{j=1}^d$ to be the immediate successors of $f_\la$.

If $f_\la$  is an \alf, $f_\la = \frac{1}{n}\sum_{j=1}^df_j$,
where $\{f_1<\cdots< f_d\}\subset W_{m-\ell-1}$, set
$\{f_j\}_{j=1}^d$ to be the immediate successors of $f_\la$.\\
{\bf Remark:} If $f_{\la^-}$ is a \schr, $f_{\la^-} =
\sum_{j=1}^df_j$ and there exists $j>1$ such that $f_\la = f_j$,
then $f_\la$ is of the form $f_\la = \frac{1}{m}\sum_{j=1}^\ell g_j$, where
$m > \max\supp f_{j-1}$. In this case, set $\{g_j\}_{j=1}^\ell$ to be
the immediate successors of $f_\la$.

It is clear that the procedure ends in at most $m+1$
steps.\vskip6pt

\begin{dfn} Let $x\in \X, f\in W$ such that $\supp f\cap \supp x\neq\varnothing$, $\{f_\la\}_{\la\in\La}$ be a
tree analysis of $f$.
\begin{enumerate}

\item[(i)] We say that $f_\mu$ covers $x$, with respect to
$\{f_\la\}_{\la\in\La}$, for some $\mu\in\La$, if $\supp
f_\mu\cap\supp x = \supp f\cap\supp x$.

\item[(ii)] We say that $f_\mu$ covers $x$ for the first time,
with respect to $\{f_\la\}_{\la\in\La}$, for some $\mu\in\La$, if
$\mu = \max\{\la\in\La: f_\la$ covers $x\}$.
\end{enumerate}
\end{dfn}

\begin{dfn} Let $x\in \X, f\in W, \{f_\la\}_{\la\in\La}$ be a
tree analysis of $f$, $\la\in \La$ be the node of $\La$ such that
$f_\la$ covers $x$ for the first time, with respect to
$\{f_\la\}_{\la\in\La}$. If $\{\mu_j\}_{j=1}^d$ are the immediate
successors of $\la$ in $\La$, $j_1 = \min\{j:\ran
f_{\mu_j}\cap\ran x \neq \varnothing\}, j_2 = \max\{j:\ran
f_{\mu_j}\cap\ran x \neq \varnothing\}$, set $x^1 =
x|_{[1,\ldots,\max\supp f_{\mu_{j_1}}]}$, $x^3 = x|_{[\min\supp
f_{\mu_{j_2}},+\infty)}$, $x^2 = x - x^1 - x^3$. Then $x^1, x^2,
x^3$ are called the initial, the middle and the final part of $x$
respectively, with respect to $\{f_\la\}_{\la\in\La}$.
\end{dfn}

\begin{rmk} If $\supp f\cap\supp x$ is not a singleton, then $x_1$
and $x_3$ are not zero and $x_1 < x_3$. However $x_2$ might be
zero.\label{remqqq1}
\end{rmk}

\begin{lem} Let $\{x_m\}_{m\inn}$ be a block sequence in
$\X$, $f\in W$, $\{f_\la\}_{\la\in\La}$ be a tree analysis of $f$
and $G = \{m\inn: \supp f\cap \supp x_m\neq\varnothing\}$. For
$m\in G$ set $\la_m, \la_m^1$ to be the nodes of $\La$ that cover
$x_m, x_m^1$ for the first time respectively and assume that
$\#\big\{\supp f_{\la_m}\cap \supp x_m\big\} > 1$, for all $m\in
G$. Then:

\begin{enumerate}
\item[(i)] $\la_m^1 > \la_m$ and $\max\supp f_{\la_m^1} <
\max\supp x_m$, for all $m\in G$.

\item[(ii)] For any $m\in G$ and $\la\geqslant\la_m^1$ such that
$\ran f_\la\cap \ran x_m^1\neq\varnothing$ and $\ran f_\la\cap
\ran x_\ell^1\neq\varnothing$, for some $\ell\neq m$, we have that
$\ell<m$ and $\la_\ell^1 > \la$.

\item[(iii)] The map $m\to\la_m^1$ is one to one.
\end{enumerate}\label{lem6}
\end{lem}

\begin{proof}
Let $m\in G$. Evidently $\la_m^1\geqslant\la_m$. Suppose that
$\la_m^1 = \la_m$. This means that $x_m^1$ and $x_m$ are covered
for the first time simultaneously, which can only be the case if
$\#\big\{\supp f_{\la_m}\cap \supp x_m\big\} = 1$. Moreover
$\max\supp f_{\la_m^1}\leqslant\max\supp x_m^1$ and by Remark
\ref{remqqq1}, we have that $\max\supp x_m^1< \max\supp x_m^3 =
\max\supp x_m$.

For the second statement, notice that since $\la\geqslant
\la_m^1$, it follows that $\max\supp f_\la\leqslant\max\supp
f_{\la_m^1} < \max\supp x_m$, therefore $\ell<m$. Moreover, since
$\supp f_\la\cap\supp x_\ell\neq\varnothing$, $\la$ is comparable
to $\la_\ell^1$. If $\la_\ell^1\leqslant\la$, then $\max\supp
f_\la\leqslant\max\supp f_{\la_\ell^1} < \max\supp x_\ell$, which
contradicts the fact that $\ran f_\la\cap \ran
x_m^1\neq\varnothing$.

The third statement follows from the second one.

\end{proof}
The next lemma is proved in exactly the same way.

\begin{lem} Let $\{x_m\}_{m\inn}$ be a block sequence in
$\X$, $f\in W$, $\{f_\la\}_{\la\in\La}$ be a tree analysis of $f$
and $G = \{m\inn: \supp f\cap \supp x_m\neq\varnothing\}$. For
$m\in G$ set $\la_m, \la_m^3$ to be the nodes of $\La$ that cover
$x_m, x_m^3$ for the first time respectively and assume that
$\#\big\{\supp f_{\la_m}\cap \supp x_m\big\} > 1$, for all $m\in
G$. Then:

\begin{enumerate}
\item[(i)] $\la_m^3 > \la_m$ and $\min\supp f_{\la_m^3} >\min\supp
x_m$, for all $m\in G$.

\item[(ii)] For any $m\in G$ and $\la\geqslant\la_m^3$ such that
$\ran f_\la\cap \ran x_m^3\neq\varnothing$ and $\ran f_\la\cap
\ran x_\ell^3\neq\varnothing$, for some $\ell\neq m$, we have that
$\ell>m$ and $\la_\ell^3 > \la$.

\item[(iii)] The map $m\to\la_m^3$ is one to one.
\end{enumerate}\label{lem7}
\end{lem}

The proof of the next lemma is even simpler and therefore it is
omitted.

\begin{lem} Let $\{x_m\}_{m\inn}$ be a block sequence in
$\X$, $f\in W$, $\{f_\la\}_{\la\in\La}$ be a tree analysis of $f$
and $G = \{m\inn: \supp f\cap \supp x_m\neq\varnothing\}$. For
$m\in G$ set $\la_m, \la_m^2$ to be the nodes of $\La$ that cover
$x_m, x_m^2$ for the first time respectively and assume that
$\#\big\{\supp f_{\la_m}\cap \supp x_m\big\} > 1$, for all $m\in
G$. Then, for any $m\in G$ with $x_m^2\neq 0$, for any
$\la\leqslant\la_m^2$ such that $\supp f_\la\cap\supp
x_\ell^2\neq\varnothing$, for some $\ell\neq m$, it follows that
$\la\leqslant\la_\ell^2$.\label{lem8}
\end{lem}

\begin{lem} Let $x_1,\ldots, x_m$, $y_1,\ldots,y_m$ be finite
normalized block sequences such that $x_i < y_{i+1}$ and $y_i <
x_{i+1}$ for $i=1,\ldots m-1$. Assume moreover that
$\{x_i\}_{i=1}^m$ and $\{y_i\}_{i=1}^m$ are both equivalent to the
usual basis of $(\mathbb{R}^m,\|\cdot\|_1)$, with a lower constant
$\theta>0$. Let $\{c_i\}_{i=1}^m$ be non negative reals with
$\sum_{i=1}^mc_i = 1$ and set $z = \sum_{i=1}^mc_ix_i$ and $w =
\sum_{i=1}^mc_iy_i$. Then:

\begin{enumerate}

\item[(i)] If $f\in W$ is an $\al$-average of size $s(f) = p$,
then there exists $g\in W$ such that $\ran g\subset \ran f\cap\ran
w$ and $\frac{1}{p}g(w)>\theta f(z) - 3\max\{c_i: i=1,\ldots,m\}$.

\item[(ii)] Let $f\in W$. Then there exists $g\in W$ with $\ran
g\subset \ran f\cap\ran w$, such that $g(w)
> \theta f(z) - 2\max\{c_i: i=1,\ldots,m\}$.

\end{enumerate}\label{lem9}
\end{lem}

\begin{proof}
For the proof of the first statement, set $i_1 = \min\{i: \ran
f\cap\ran x_i \neq \varnothing\}, i_2 = \max\{i: \ran f\cap\ran
x_i \neq \varnothing\}$. By Lemma \ref{lem4}, we conclude that
\begin{equation}
f(z) < \frac{1}{p}\sum_{i=i_1}^{i_2}c_i + 2\max\{c_i:
i=1,\ldots,m\}\label{lem9eq1}
\end{equation}

Since $\|\sum_{i=i_1 +1}^{i_2-1}c_iy_i\|
> \theta\sum_{i=i_1}^{i_2}c_i - 2\max\{c_i: i=1,\ldots,m\}$, we may choose $g\in W$
such that
\begin{equation}
g(\sum_{i=i_1 +1}^{i_2-1}c_iy_i) > \theta\sum_{i=i_1}^{i_2}c_i -
2\max\{c_i: i=1,\ldots,m\}\label{lem9eq2}
\end{equation}

We may clearly assume that $\ran g\subset \ran\big\{\cup_{i=i_1 +
1}^{i_2-1}\ran y_i\big\}\subset \ran f\cap\ran w$.  Finally,
combining \eqref{lem9eq1} and \eqref{lem9eq2}, and doing some easy
calculations we conclude that $g$ is the desired functional.

To prove the second statement, define $i_1, i_2$ as before. Then,
one evidently has that $\|\sum_{i=i_1 +1}^{i_2-1}c_iy_i\|
> \theta\sum_{i=i_1}^{i_2}c_i - 2\max\{c_i: i=1,\ldots,m\}$, therefore there exists $g\in W$
such that
\begin{equation}
g(\sum_{i=i_1 +1}^{i_2-1}c_iy_i) > \theta\sum_{i=i_1}^{i_2}c_i -
2\max\{c_i: i=1,\ldots,m\}\label{lem9eq3}
\end{equation}

It is also clear that

\begin{equation}
f(z) \leqslant \sum_{i=i_1}^{i_2}c_i \label{lem9eq4}
\end{equation}

As previously, we may assume that $\ran g\subset ran f\cap\ran w$.
Combining \eqref{lem9eq3} and \eqref{lem9eq4} we conclude the
desired result.

\end{proof}

For $\{x_m\}_{m\inn}, \{y_m\}_{m\inn}$ normalized block sequences
in $\X$ both generating $\ell_1^n$ spreading models, we
appropriately block both sequences in the same manner to obtain
further seminormalized block sequences $\{z_m\}_{m\inn}$ and
$\{w_m\}_{m\inn}$. For $f$ a given functional in $W$, we decompose
$z_m$ into $z_m^1, z_m^2, z_m^3$ its initial, middle and final
part, as previously described. Next, we proceed to construct $g^1,
g^2$ $g^3$ functionals in $W$, such that each $g^i$ acting on
$w_m$, pointwise dominates $f$ acting on $z^i_m$, for $i=1,2,3$.
The choice of the functionals $g^i$, $i=1,2,3$ is presented in the
following three lemmas.

\begin{lem}
Let $\{x_m\}_{m\inn}, \{y_m\}_{m\inn}$ be normalized block
sequences in $\X$, both generating $\ell_1^n$ spreading models,
with a lower constant $\theta>0$, such that $x_m < y_{m+1}$ and
$y_m < x_{m+1}$ for all $m\inn$. Let $\{F_m\}_{m\inn}$ be
successive subsets of the naturals, $\{c_i\}_{i\inn}$ be non
negative reals and $\{\e_m\}_{m\inn}, \{\de_m\}_{m\inn}$ be
positive reals satisfying the following:
\begin{enumerate}

\item[(i)] $F_m\in\Sn$ and $z_m = \sum_{i\in F_m}c_ix_i, w_m =
\sum_{i\in F_m}c_iy_i$ are both $(n,\e_m)$ s.c.c. for all $m\inn$.

\item[(ii)] $\max\supp z_m\Big(\frac{1}{\min\supp z_{m+1}} +
6\e_{m+1}\Big) < \frac{\de_{m+1}}{4}$, for all $m\inn$.

\end{enumerate}
Let also $f\in W$, with a tree analysis $\{f_\la\}_{\la\in\La}$
and $z_m^1$ be the initial part of $z_m$ with respect to
$\{f_\la\}_{\la\in\La}$, for all $m\inn$. Then there exists
$g^1\in W$, such that
\begin{equation*}
g^1(w_m) > 2\theta f(z^1_m) - 5\de_m,\;\text{for all}\; m\inn.
\end{equation*}\label{prop10}
\end{lem}

\begin{proof}
Let $f\in W$. We may assume that $f(e_j)\geqslant 0$, for all
$j\inn$, that $\supp f\subset \cup_{m\inn}\supp z_m$ and that
$e_j^*(z_m)\geqslant 0, e_j^*(w_m)\geqslant 0$ for all $j,k\inn$.
Set $G = \{m\inn:\supp f\cap\supp x_m\neq\varnothing\}$.

We may assume that for any $m\in G$, $\supp f\cap\supp z_m$ is not
a singleton. Otherwise there exists $f^\prime\in W$ that satisfies
this condition for $G^\prime = \{m\inn:\supp f^\prime\cap\supp
z_m\neq\varnothing\}$ and $f^\prime(z_m) \geqslant f(z_m) - \e_m$,
for all $m\inn$.

Let $\{f_\lambda\}_{\lambda\in\Lambda}$ be a tree analysis of $f$.
Denote by $z_m^1$ the initial part of $z_m$ and $\la_m^1$ the node
of $\La$ that cover $z_m^1$ for the first time, for all $m\in G$,
all with respect to $\{f_\lambda\}_{\lambda\in\Lambda}$.

We proceed to the construction of $g^1$. Set
\begin{eqnarray*}
\mathcal{C}_m^1\!\!\!\!\! &=&\!\!\!\!\! \big\{\la\in\La:
\la\geqslant\la_m^1, \min\{\supp f_\la\cap\supp z_m^1\} =
\min\{\supp
f_{\la_m^1}\cap\supp z_m^1\}\big\}\\
&&\cup\{\la\in\La:\la\leqslant \la_m^1\}
\end{eqnarray*}
Notice that $\mathcal{C}_m^1$ is a maximal chain in $\La$. Set
\begin{eqnarray*}
\nu_m^1\!\!\! &=&\!\!\! \max\{\la\in\mathcal{C}_m^1: \ran
f_\la\cap\ran
z_\ell^1\neq\varnothing, \text{for some}\;\ell\neq m\}\\
\mu_m\!\!\! &=&\!\!\! \min\{\la\in\mathcal{C}_m^1: \la\geqslant
\la_m^1, f_\la\; \text{is an}\; \al\text{-average and there
exists}\;\be\in\scc(\la)\\
&&\!\!\!\text{such that}\;\ran f_\be\cap\ran
z_m^1\neq\varnothing\;\text{and}\; \ran f_\be\cap\ran z_\ell^1 =
\varnothing\;\text{for}\;\ell\neq m\}
\end{eqnarray*}
where $\scc(\la)$ are the immediate successors of $\la$ in $\La$.

\begin{clm}If for some $m\in G$ we have that
$\la_m^1\leqslant \nu_m^1 <\mu_m$, then $\nu_m^1 = \mu_m^-$.
\end{clm}

\begin{proof}[Proof of claim]

First notice that in this case $f_{\nu_m^1}$ must be a \schr. If
$f_{\nu_m^1}$ were an $\al$-average, then if we denote its
immediate successor in $\mathcal{C}_m^1$ by $\be$, then $\ran
f_\be\cap\ran z_m^1\neq\varnothing$ and $\ran f_\be\cap\ran
z_\ell^1 = \varnothing$ for $\ell\neq m$, therefore $\mu_m^1$
would not be the minimal element satisfying this condition, since
we assumed that $\nu_m^1 <\mu_m$, a contradiction. Since
$f_{\nu_m^1}$ is a \schr, it follows that if we denote its
immediate successor in $\mathcal{C}_m^1$ by $\be$, then $f_\be$ is
an $\al$-average, such that $\ran f_\be\cap\ran
z_m^1\neq\varnothing$ and $\ran f_\be\cap\ran z_\ell^1 =
\varnothing$ for $\ell\neq m$. Since $\nu_m^1<\mu_m$, it follows
that $\be = \mu_m$.

\end{proof}

Set
\begin{eqnarray*}
\La_1 &=& \{\la\in\La:\;\text{there exists}\;m\in G\;\text{such
that}\;\la\leqslant\la_m^1\}\\
&&\cup\{\la\in\La:\;\text{there exists}\;m\in G\;\text{such
that}\;\la\leqslant\mu_m\;\text{and}\;\nu_m^1\geqslant \la_m^1\}
\end{eqnarray*}
For $\la\in\La_1$, set
\begin{equation*}
G_\la = \{m\in G: \la\leqslant\la_m^1\}\cup\{m\in G:
\la\leqslant\mu_m\;\text{and}\;\nu_m^1\geqslant\la_m^1\}
\end{equation*}
For every $\la\in\La_1$, we will inductively construct $g_\la^1\in
W$ satisfying the following.

\begin{enumerate}

\item[(i)] $g_\la^1(w_m) > \theta f_\la(z_m^1) - 4\de_m$, for all
$m\in G_\la$.

\item[(ii)] $\ran g_\la^1\subset \ran
f_\la\cap\ran\big\{\cup\{\ran w_m : m\in G_\la\}\big\}$.

\item[(iii)] If $f_\la$ is an $\al$-average, then so is $g_\la^1$
and $s(g_\la^1) = s(f_\la)$.

\end{enumerate}

Before proceeding to the construction, we would like to stress out
that Lemma \ref{lem6} assures us that whenever a functional
$f_\la$, $\la\in\La_1$ acts on more than one vectors $z_k^1$, then
all vectors except for the rightmost one, have been covered for
the first time in a previous step. Therefore in this case, we are
free to focus the inductive step on one vector. In particular, if
$\la\in\La, \la\geqslant\la_m^1$ for some $m\in G$, such that
$\ran f_\la\cap\ran z_m^1\neq\varnothing$ and $\ran f_\la\cap\ran
z_\ell^1\neq\varnothing$ for $\ell\neq m$, then besides the fact
that $\ell<m$ and $\la_\ell^1\geqslant\la$, it also follows that
$\la\in\mathcal{C}_m^1$ (as well as $\la\in\mathcal{C}_\ell^1$).

Let $\la\in\La_1$. We distinguish six cases. The first inductive
step falls under the first two.\vskip3pt

\noindent {\em Case 1:} There exists $m\in G$ such that $\la =
\la_m^1 = \mu_m$ and $\nu_m < \la_m^1$.

In this case $f_\la$ is an $\al$-average, $f_\la =
\frac{1}{p}\sum_{j=1}^d f_{\be_j}$, where $\{\be_j\}_{j=1}^d$ are
the immediate successors of $\la$. By Lemma \ref{lem9}, there
exists $g\in G$, such that $\ran g\subset\ran f_\la\cap\ran w_m$
and $\frac{1}{p}g(w_m) > \theta f_\la(z_m^1) - 3\max\{c_i: i\in
F_m\}$. Set $g_\la^1 = \frac{1}{p}g$. Since $\max\{c_i: i\in F_m\}
< \e_m < \de_m$, we conclude that $g_\la^1$ satisfies the
inductive assumption.\vskip3pt

\noindent {\em Case 2:} There exists $m\in G$ such that $\la =
\la_m^1 < \mu_m$ and $\nu_m < \la_m^1$.

Then $f_\la$ is a \schr, $f_\la = \sum_{j=1}^df_{\be_j}$. Then
again by Lemma \ref{lem9}, there exists $g\in W$ such that $\ran
g\subset\ran f_\la\cap\ran w_m$ and $g(w_m) > \theta f(z_m^1) -
3\max\{c_i: i\in F_m\}$. Set $g_\la^1 = g$. As in the previous
case, we conclude that $g_\la^1$ satisfies the inductive
assumption.\vskip3pt

\noindent {\em Case 3:} For any $m\in G$ such that $\ran
f_\la\cap\ran z_m^1\neq\varnothing$, we have that $\la<\la_m^1$.

Then if $f_\la = \sum_{j=1}^df_{\be_j}$ (or $f_\la =
\frac{1}{p}\sum_{j=1}^df_{\be_j}$), for $j=1,\ldots,d$ there exist
$g_{\be_j}^1$, already satisfying the inductive assumption. Then
it is easy to see that $g_\la^1 = \sum_{j=1}^dg^1_{\be_j}\in W$
(or $g^1_\la = \frac{1}{p}\sum_{j=1}^dg^1_{\be_j}\in W$) and is
the desired functional.\vskip3pt

\noindent {\em Case 4:} There exists $m\in G$ such that
$\la>\mu_m$.

Since $\la\in\La_1$, there exists at least one $\ell<m$ in $G$,
such that $\la<\la_\ell^1$. If $f_\la = \sum_{j=1}^df_{\be_j}$ (or
$f_\la = \frac{1}{p}\sum_{j=1}^df_{\be_j}$), set $j_0 =
\max\{j:$there exists $\ell<k$ such that $\ran f_{\be_j}\cap\ran
z_\ell^1\neq\varnothing\}$. Then it is easy to see that $g_\la^1 =
\sum_{j=1}^{j_0}g^1_{\be_j}\in W$ (or $g^1_\la =
\frac{1}{p}\sum_{j=1}^{j_0}g^1_{\be_j}\in W$) and satisfies the
inductive assumption.\vskip3pt

\noindent {\em Case 5:} There exists $m\in G$ such that $\la =
\mu_m$ and $\la_m^1\leqslant\nu_m^1$.

This both covers the case when $\mu_m = \la_m^1$ and when $\mu_m >
\la_m^1$. The claim yields that in either case $\nu_m^1\geqslant
\mu_m^-$.

If $\nu_m^1 = \mu_m^-$, simply repeat what was done in case 1.
Otherwise, $\nu_m^1\geqslant\mu_m$ and there exist at least one
$\ell<m$ in $G$, such that $\ran f_\la\cap\ran
z_\ell^1\neq\varnothing$. If $f_\la =
\frac{1}{p}\sum_{j=1}^df_{\be_j}$, set $j_0 = \max\{j:$ there
exists $\ell<m$ such that $\ran f_{\be_j}\cap\ran
z_\ell^1\neq\varnothing\}$. Since $\la = \mu_m$, we have that $j_0
< d$. Apply Lemma \ref{lem9} and find $g\in W, \ran g\subset\ran
f_\la\cap\ran w_m$ such that $\frac{1}{p}g(w_k) > \theta
f_\la(z_k^1) - 3\max\{c_i: i\in F_m\}$. Set $g_\la^1 =
\frac{1}{p}\sum_{j=1}^{j_0}g_{\be_j}^1 + \frac{1}{p}g$. Then
$g_\la^1\in W$ and satisfies the inductive assumption. In
particular, note that $g_\la^1(w_m) > \theta f_\la(z_m^1) -
3\de_m$.\vskip3pt

\noindent {\em Case 6:} There exists $m\in G$, such that
$\la_m^1\leqslant\la<\mu_m$ and $\nu_m^1\geqslant \la_m^1$.

We will prove by induction on $q = |\la| - |\mu_m|$ that there
exists $g_\la^1\in W$ satisfying conditions (i), (ii) and (iii)
from our initial inductive assumption and moreover a stronger
version of condition (i). In particular:\vskip3pt

\noindent If $f_\la$ is an $\al$-average, then $g_\la^1(w_m) >
\theta f_\la(z_m^1) - 3\de_m - \frac{\de_m}{4}$.\vskip3pt

\noindent If $f_\la$ is a \schr, then $g_\la^1(w_m) > \theta
f_\la(z_m^1) - 3\de_m - \frac{\de_m}{2}$.\vskip3pt

For convenience start the induction for $q = 0$, i.e. $\la =
\mu_m$. As we have noted in this case $g_\la^1(w_m) > \theta
f_\la(z_m^1) - 3\de_m$.

Assume that it is true for some $q< |\la_m^1| - |\mu_m|$. Then for
$\la$ such that $|\la| - |\mu_m| = q + 1$, the claim yields that
$\nu_m^1\geqslant\la$.

If $f_\la$ is an $\al$-average, $f_\la =
\frac{1}{p}\sum_{j=1}^df_{\be_j}$, since $\la\leqslant \nu_m^1$,
$\la<\mu_m$, we have that $\ran f_{\be_d}\cap\ran
z_\ell^1\neq\varnothing$, for some $\ell<m$. Therefore $\ran
f_{\be_j}\cap\ran z_m^1=\varnothing$ for $j<d$ and there exists
$g_{\be_d}^1$ satisfying the stronger inductive assumption. Set
$g_\la^1 = \frac{1}{p}\sum_{j=1}^dg_{\be_j}^1$. As always
$g_\la^1\in W$ and it satisfies the initial inductive assumption.
It also satisfies the stronger one. Indeed, $g_\la^1(w_m) =
\frac{1}{p}g_{\be_d}^1(w_m) > \frac{1}{p}\big(\theta
f_{\be_d}(z_m^1) - 3\de_m - \frac{\de_m}{2}\big) =
\frac{1}{p}\big(p\theta f_\la(z_m^1) - 3\de_m -
\frac{\de_m}{2}\big) = \theta f_\la(z_m^1) - \frac{3\de_m}{p} -
\frac{\de_m}{2p} > \theta f_\la(z_m^1) - 3\de_m -
\frac{\de_m}{4}$.

If $f_\la$ is a \schr, $f_\la = \sum_{j=1}^df_{\be_j}$, since
$\la\leqslant \nu_m^1$ we have that $\ran f_{\be_j}\cap\ran
z_\ell^1\neq\varnothing$, for some $\ell<m$ and some $j\leqslant
d$. Set $j_0 = \min\{j:\ran f_{\be_j}\cap\ran
z_m^1\neq\varnothing\}$. Therefore $\ran f_{\be_j}\cap\ran
z_m^1=\varnothing$ for $j<j_0$ and there exists an $\al$-average
$g_{\be_{j_0}}^1$ satisfying the stronger inductive assumption.

Choose $\{J_r\}_{r=1}^{r_0}$ successive subsets of the naturals
satisfying the following.
\begin{enumerate}

\item[(i)] $\cup_{r=1}^{r_0}J_r = \{j: j_0 < j\leqslant d\}$

\item[(ii)] $\{\min\supp f_j: j\in J_r\}$ is a maximal
$\mathcal{S}_{n-1}$ set for $r<r_0$ and $\{\min\supp f_j: j\in
J_{r_0}\}\in\mathcal{S}_{n-1}$

\end{enumerate}
We conclude that $r_0 \leqslant\max\supp z_{m-1}$. Moreover, Lemma
\ref{lemqqq6} yields that for $r\leqslant r_0$
\begin{equation*}
\sum_{j\in J_r}f_j(z_m^1) < \frac{1}{\min\supp z_m} + 6\e_m\
\end{equation*}
Assumption (ii) of the proposition yields that
$\sum_{j>j_0}f_{\be_j}(z_m^1) < \frac{\de_m}{4}$.

Set $g_\la^1 = \sum_{j=1}^{j_0}g_{\be_j}^1$. Then $g_\la^1(w_m) =
g_{\be_{j_0}}^1(w_m) > \theta f_{\be_{j_0}}(z_m^1) - 3\de_m -
\frac{\de_m}{4} = \theta\sum_{j=1}^{j_0}f_{\be_j}(z_m^1) - 3\de_m
- \frac{\de_m}{4} > \theta f_\la(z_m^1) - \frac{\de_m}{4} - 3\de_m
- \frac{\de_m}{4} = \theta f_\la(z_m^1) - 3\de_m -
\frac{\de_m}{2}$. This ends the inductive step in case 6 and also
the initial induction.

Set $g^1 = g^1_\varnothing$. Then:
\begin{equation*}
g^1(w_m) > \theta f(z_m^1) - 4\de_m,\;\text{for all}\;m\in G
\end{equation*}

Lifting the restriction that for any $m\in G$, $\supp f\cap\supp
z_m$ is not a singleton, in the general case we conclude that
$g^1(w_m)
> \theta f(z_m^1) - 5\de_m,\;\text{for all}\;m\in G$

\end{proof}

\begin{lem}
Let $\{x_m\}_{m\inn}, \{y_m\}_{m\inn}$ be normalized block
sequences in $\X$, both generating $\ell_1^n$ spreading models,
with a lower constant $\theta>0$, such that $x_m < y_{m+1}$ and
$y_m < x_{m+1}$ for all $m\inn$. Let $\{F_m\}_{m\inn}$ be
successive subsets of the naturals, $\{c_i\}_{i\inn}$ be non
negative reals and $\{\e_m\}_{m\inn}, \{\de_m\}_{m\inn}$ be
positive reals satisfying the following:
\begin{enumerate}

\item[(i)] $F_m\in\Sn$ and $z_m = \sum_{i\in F_m}c_ix_i, w_m =
\sum_{i\in F_m}c_iy_i$ are both $(n,\e_m)$ s.c.c. for all $m\inn$.

\item[(ii)] $\max\supp z_m\Big(\frac{1}{\min\supp z_{m+1}} +
6\e_{m+1}\Big) < \frac{\de_{m+1}}{4}$, for all $m\inn$.

\end{enumerate}
Let also $f\in W$, with a tree analysis $\{f_\la\}_{\la\in\La}$
and $z_m^3$ be the final part of $z_m$ with respect to
$\{f_\la\}_{\la\in\La}$, for all $m\inn$. Then there exists
$g^3\in W$, such that
\begin{equation*}
g^3(w_m) > \frac{\theta}{2} f(z^3_m) - 3\de_m,\;\text{for all}\;
m\inn.
\end{equation*}\label{lemqqq11}
\end{lem}

\begin{proof}
Let $f\in W$. As in the previous proof, assume that
$f(e_j)\geqslant 0$, for all $j\inn$, that $\supp f\subset
\cup_{m\inn}\supp z_m$ and that $e_j^*(z_m)\geqslant 0,
e_j^*(w_m)\geqslant 0$ for all $j,k\inn$. Set $G = \{m\inn:\supp
f\cap\supp x_m\neq\varnothing\}$.

Assume again that for any $m\in G$, $\supp f\cap\supp z_m$ is not
a singleton. Otherwise there exists $f^\prime\in W$ that satisfies
this condition for $G^\prime = \{m\inn:\supp f^\prime\cap\supp
z_m\neq\varnothing\}$ and $f^\prime(z_m) \geqslant f(z_m) - \e_m$,
for all $m\inn$.

Let $\{f_\lambda\}_{\lambda\in\Lambda}$ be a tree analysis of $f$.
Denote by $z_m^3$ the final part of $z_m$ and $\la_m^3$ the node
of $\La$ that cover $z_m^3$ for the first time, for all $m\in G$,
all with respect to $\{f_\lambda\}_{\lambda\in\Lambda}$.

Set
\begin{eqnarray*}
\mathcal{C}_m^3\!\!\! &=&\!\!\! \big\{\la\in\La:
\la\geqslant\la_m^3, \max\{\supp f_\la\cap\supp z_m^3\} =
\max\{\supp
f_{\la_m^3}\cap\supp z_m^3\}\big\}\\
&&\cup\{\la\in\La:\la\leqslant \la_m^3\}\\
\nu_m^3\!\!\! &=&\!\!\! \max\{\la\in\mathcal{C}_m^3: \ran
f_\la\cap\ran z_\ell^3\neq\varnothing, \text{for some}\;\ell\neq
m\}
\end{eqnarray*}
Set
\begin{equation*}
\La_3 = \{\la\in\La:\;\text{there exists}\;m\in G\;\text{such
that}\;\la\leqslant\la_m^3\}
\end{equation*}
For every $\la\in\La_3$, we will inductively construct $g_\la^3\in
W$ satisfying the following.
\begin{enumerate}

\item[(i)] $g_\la^3(w_m) > \frac{\theta}{2} f_\la(z_m^3) -
2\de_m$, for all $m\in G$ such that $\la_m^3\geqslant\la$.

\item[(ii)] $\ran g_\la^3\subset \ran
f_\la\cap\ran\big\{\cup\{\ran w_m : \la_m^3\geqslant\la\}\big\}$.

\item[(iii)] If $f_\la$ is an $\al$-average, then so is $g_\la^3$
and $s(g_\la^3) = s(f_\la)$.

\end{enumerate}

Just as in the construction of $g^1$, Lemma \ref{lem7} assures us
that whenever a functional $f_\la$, $\la\in\La_2$ acts on more
than one vectors $z_m^3$, then all vectors except for the leftmost
one, have been covered for the first time in a previous step.

Let $\la\in\La_3$. We distinguish 4 cases, the first inductive
step falls under the first case.\vskip3pt

\noindent {\em Case 1:} There exists $m\in G$, such that $\la =
\la_m^3$ and $\nu_m^3 < \la_m^3$.

If $f_\la$ is an $\al$-average, $f_\la =
\frac{1}{p}\sum_{j=1}^df_{\be_j}$, by Lemma \ref{lem9} there
exists $g\in W$ such that $\ran g\subset\ran f_\la\cap\ran w_m$
and $\frac{1}{p}g(w_m) > \theta f_\la(z_m^3) - 3\max\{c_i: i\in
F_m\}$. Set $g_\la^3 = \frac{1}{p}g$.

If $f_\la$ is a \schr, then by Lemma \ref{lem9} there exists $g\in
W$ such that $\ran g\subset\ran f_\la\cap\ran w_m$ and $g(w_m) >
f_\la(z_m^3) - 2\max\{c_i: i\in F_m\}$. Set $g_\la^3 =
g$.\vskip3pt

\noindent {\em Case 2:} For any $m\in G$ such that $\ran
f_\la\cap\ran z_m^3\neq\varnothing$, we have that $\la < \la_m^3$.
If $f_\la = \sum_{j=1}^df_{\be_j}$ (or $f_\la =
\frac{1}{p}\sum_{j=1}^df_{\be_j}$), set $g_\la^3 =
\sum_{j=1}^dg_{\be_j}^3$ (or $g_\la^3 =
\frac{1}{p}\sum_{j=1}^dg_{\be_j}^3$).\vskip3pt

\noindent {\em Case 3:} There exists $m\in G$, such that
$\la>\la_m^3$.

Since $\la\in\La_3$, there exists at least one $\ell>m$ such that
$\la_\ell^3>\la$. If $f_\la = \sum_{j=1}^df_{\be_j}$ (or $f_\la =
\frac{1}{p}\sum_{j=1}^df_{\be_j}$), set $j_0 = \min\{j: \ran
f_{\be_j}\cap\ran z_\ell^3\neq\varnothing$, for some $\ell>k\}$.
Set $g_\la^3 = \sum_{j=j_0}^dg_{\be_j}^3$ (or $g_\la^3 =
\frac{1}{p}\sum_{j=j_0}^dg_{\be_j}^3$).\vskip3pt

\noindent {\em Case 4:} There exists $m\in G$, such that $\la =
\la_m^3$ and $\nu_m^3\geqslant\la_m^3$.

If $f_\la$ is an $\al$-average, $f_\la =
\frac{1}{p}\sum_{j=1}^df_{\be_j}$, set $j_0 = \min\{j: \ran
f_{\be_j}\cap\ran z_\ell^3\neq\varnothing$ for some $\ell>m\}$.
Then $j_0 > 1$, otherwise $z_m^3$ would have been covered for the
first time in a previous step. By Lemma \ref{lem9} there exists
$g\in W$ such that $\ran g\subset\ran f_\la\cap\ran w_m$ and
$\frac{1}{p}g(w_m) > \theta f_\la(z_m^3) - 3\max\{c_i: i\in
F_m\}$. Set $g_\la^3 = \frac{1}{p}g +
\frac{1}{p}\sum_{j=j_0}^dg_{\be_j}^3$.

If $f_\la$ is a \schr, $f_\la = \sum_{j=1}^df_{\be_j}$, set again
$j_0 = \min\{j: \ran f_{\be_j}\cap\ran z_\ell^3\neq\varnothing$
for some $\ell>m\}$ and as before $j_0>1$. By Lemma \ref{lem9}
there exists $g\in W$ such that $\ran g\subset\ran(\sum_{j <
j_0}f_{\be_j})\cap\ran w_m$ and $g(w_m) > \theta \sum_{j <
j_0}f_{\be_j}(z_m^3) - 2\max\{c_i: i\in F_m\}$. Since $j_0>1$, it
follows that $s(f_{\be_{j_0}}) > \min\supp z_m$. By Lemma
\ref{lem4} and assumption (ii) of the proposition, we conclude the
following.
\begin{equation*}
f_{\be_{j_0}}(z_m^3) < \frac{1}{\min\supp z_m} + 2\max\{c_i: i\in
F_m\} < \frac{1}{\min\supp z_m} + 2\e_m < \frac{\de_m}{4}
\end{equation*}
Set $g_\la^3 = \frac{1}{2}g + \sum_{j=j_0}^dg_{\be_j}^3$. Then
$g_\la(w_m) > \frac{\theta}{2}f_\la(z_m^3) - 2\de_m$.\vskip3pt

This ends induction. Set $g^3 = g^3_\varnothing$. Then:
\begin{equation*}
g^3(w_m) > \frac{\theta}{2} f(z_m^3) - 2\de_m,\;\text{for
all}\;k\in G
\end{equation*}

Lifting the restriction that for any $m\in G$, $\supp f\cap\supp
z_m$ is not a singleton, in the general case we conclude that
$g^3(w_m)
> \frac{\theta}{2} f(z_m^3) - 3\de_m,\;\text{for all}\;m\in G$.

\end{proof}

\begin{lem}
Let $\{x_m\}_{m\inn}, \{y_m\}_{m\inn}$ be normalized block
sequences in $\X$, both generating $\ell_1^n$ spreading models,
with a lower constant $\theta>0$, such that $x_m < y_{m+1}$ and
$y_m < x_{m+1}$ for all $m\inn$. Let $\{F_m\}_{m\inn}$ be
successive subsets of the naturals, $\{c_i\}_{i\inn}$ be non
negative reals and $\{\e_m\}_{m\inn}, \{\de_m\}_{m\inn}$ be
positive reals satisfying the following:
\begin{enumerate}

\item[(i)] $F_m\in\Sn$ and $z_m = \sum_{i\in F_m}c_ix_i, w_m =
\sum_{i\in F_m}c_iy_i$ are both $(n,\e_m)$ s.c.c. for all $m\inn$.

\item[(ii)] $\max\supp z_m\Big(\frac{1}{\min\supp z_{m+1}} +
6\e_{m+1}\Big) < \frac{\de_{m+1}}{4}$, for all $m\inn$.

\end{enumerate}
Let also $f\in W$, with a tree analysis $\{f_\la\}_{\la\in\La}$
and $z_m^2$ be the middle part of $z_m$ with respect to
$\{f_\la\}_{\la\in\La}$, for all $m\inn$. Then there exists
$g^2\in W$, such that
\begin{equation*}
g^2(w_m) > \frac{\theta}{2} f(z^2_m) - 5\de_m,\;\text{for all}\;
m\inn.
\end{equation*}\label{lemqqq12}
\end{lem}

\begin{proof}
Let $f\in W$. As usually, assume that $f(e_j)\geqslant 0$, for all
$j\inn$, that $\supp f\subset \cup_{m\inn}\supp z_m$ and that
$e_j^*(z_m)\geqslant 0, e_j^*(w_m)\geqslant 0$ for all $j,k\inn$.
Set $G = \{m\inn:\supp f\cap\supp x_m\neq\varnothing\}$.

Assume again that for any $m\in G$, $\supp f\cap\supp z_m$ is not
a singleton. Otherwise there exists $f^\prime\in W$ that satisfies
this condition for $G^\prime = \{m\inn:\supp f^\prime\cap\supp
z_m\neq\varnothing\}$ and $f^\prime(z_m) \geqslant f(z_m) - \e_m$,
for all $m\inn$.

Let $\{f_\lambda\}_{\lambda\in\Lambda}$ be a tree analysis of $f$.
Denote by $z_m^2$ the middle part of $z_m$ and $\la_m^2$ the node
of $\La$ that cover $z_m^2$ for the first time, for all $m\in G$,
all with respect to $\{f_\lambda\}_{\lambda\in\Lambda}$.

Set
\begin{equation*}
\La_2 = \{\la\in\La:\;\text{there exists}\;m\in G\;\text{such
that}\;z_m^2\neq 0\;\text{and}\;\la\leqslant\la_m^2\}
\end{equation*}

For every $\la\in\La_2$, we will inductively construct $g_\la^2\in
W$ such that:
\begin{enumerate}

\item[(i)] $g_\la^2(w_m) > \frac{\theta}{2} f_\la(z_m^2) -
4\de_m$, for all $m\in G$.

\item[(ii)] $\ran g_\la^2\subset \ran f_\la$.

\item[(iii)] If $f_\la$ is an $\al$-average, then so is $g_\la^2$
and $s(g_\la^2) = s(f_\la)$.

\end{enumerate}

By Lemma \ref{lem8} it follows that whenever $\la\in\La_2$ such
that $f_\la\cap\ran z_m^2$, for some $m$, then
$\la\leqslant\la_m^2$. Therefore, although it might be the case
that $f_\la$ covers many $z_m^2$ for the first time
simultaneously, it cannot act on any $z_m^2$ without covering it.

The first inductive step it similar to the general one, therefore
let $\la\in\La_2$ and assume that the inductive assumption holds
for any $\mu>\la$.\vskip3pt

\noindent {\em Case 1:} $f_\la$ is an $\al$-average.

Set
\begin{equation*}
D = \{m\in G: \la = \la_m^2\},\quad E = \{m\in G: \la < \la_m^2\}
\end{equation*}
If $f_\la = \frac{1}{p}\sum_{j=1}^df_{\be_j}$, set
\begin{equation*}
H = \{j: \ran f_{\be_j}\cap\ran z_m^2\neq\varnothing\;\text{for
some}\;m\in E\}
\end{equation*}
As we have noted, $\ran f_{\be_j}\cap\ran z_m^2 = \varnothing$,
for any $j\in H, m\in D$.

For $m\in D$, since $\la = \la_m^2$, there exists at least one
$j_m$, such that $\ran f_{\be_{j_m}}\cap\ran z_\ell^2 =
\varnothing$ for any $\ell\neq m$, in fact there exist
$j_{m_1}<j_{m_2}$ such that $\ran f_{\be_{j_{m_i}}}\subset\ran
z_m^2$, for $i=1,2$. Therefore $\#H < p - \#D$.

For $m\in D$ apply Lemma \ref{lem9} and find $g_m\in W$, such that
$\ran g_m\subset\ran f_\la\cap\ran w_m$ and
$\frac{1}{p}g(w_m)\geqslant\theta f_\la(z_2^m) - 3\max\{c_i: i\in
F_m\}$. We may assume that $\ran g\subset \ran z_m^2$ (to see this
restrict $f_\la$ to the range of $z_m^2$).

Set $g_\la^2 = \frac{1}{p}\sum_{m\in D}g_m + \frac{1}{p}\sum_{j\in
H}g_{\be_j}^2$. By the above it follows that $g_\la^2\in W$ and
that it satisfies the inductive assumption.\vskip3pt

\noindent {\em Case 2:} $f_\la$ is a \schr.

\begin{equation*}
D = \{m\in G: \la = \la_m^2\},\quad E = \{m\in G: \la < \la_m^2\}
\end{equation*}
If $f_\la = \sum_{j=1}^df_{\be_j}$, set
\begin{equation*} H = \{j:
\ran f_{\be_j}\cap\ran z_m^2\neq\varnothing\;\text{for some}\;m\in
E\}
\end{equation*}
Again, $\ran f_{\be_j}\cap\ran z_m^2 = \varnothing$, for any $j\in
H, m\in D$.

Set $m_1 = \min\{m: \ran f_\la\cap\ran z_m^2\neq\varnothing\}$.
Let $m\in D, m>m_1$. Set $j_m = \min\{j: \ran f_{\be_j}\cap\ran
z_m^2\neq\varnothing\}$. Then $\ran f_{\be_{j_m}}\subset\ran
z_m^2$.

By applying Lemma \ref{lem9} find $g_m\in W$ an $\al$-functional
of size $s(g_m) = s(f_{\be_{j_m}})$ such that $\ran g_m\subset\ran
f_{\be_{j_m}}\cap\ran w_m$ and $g_m
> \theta f_{\be_{j_m}}(z_m^2) - 3\max\{c_i: i\in
F_m\}$. By the fact that $\{f_{\be_j}\}_{j=1}^d$ is admissible and
very fast growing, just as in case 6 of the proof of Lemma
\ref{prop10}, it follows that $\sum_{j>j_m}f_{\be_j}(z_m^2) <
\frac{\de_m}{4}$.

If $\min D > m_1$, set $g_\la^2 = \sum_{j\in H}g_{\be_j}^2 +
\sum_{m\in D} g_m$.

If $\min D = m_1$, set $j_0 = \max\{j: \ran f_{\be_j}\cap\ran
z_{m_1}^2\neq\varnothing\}$. Just as in case 4 of the proof of
Lemma \ref{lemqqq11}, find $g_{m_1}\in W$, such that $\ran
g_{m_1}\subset\ran(\sum_{j<j_0}f_{\be_j})\cap\ran w_{m_1}$ and
$g_{m_1}(w_{m_1})
> \theta \sum_{j<j_0}f_{\be_j}(z_{m_1}^2) - 2\max\{c_i: i\in
F_m\}$. Again we have that $f_{\be_{j_0}}(z_{m_1}^2) <
\frac{\de_m}{4}$. Set $g_\la^2 = \frac{1}{2}g_{m_1} + \sum_{j\in
H}g_{\be_j}^2 + \sum_{m\in D\setminus\{m_1\}} g_m$.\vskip3pt

The inductive construction is complete. Set $g^2 =
g^2_\varnothing$. Then:
\begin{equation*}
g^2(w_m) > \frac{\theta}{2} f(z_m^2) - 4\de_m,\;\text{for
all}\;m\in G
\end{equation*}

Lifting the restriction that for any $m\in G$, $\supp f\cap\supp
z_m$ is not a singleton, in the general case we conclude that
$g^2(w_m)
> \frac{\theta}{2} f(z_m^2) - 5\de_m,\;\text{for all}\;m\in G$.

\end{proof}

We are now ready to prove the main result of this section.

\begin{proof}[Proof of Proposition \ref{cor11}]
Fix $\theta > 0$ such that both $\{x_m\}_{m\inn}$ and
$\{y_m\}_{m\inn}$ generate $\ell_1^n$ spreading models with a
lower constant $\theta$. Fix $\{\de_m\}_{m\inn}$ a sequence of
positive reals, such that $\sum_{m=1}^\infty\de_m <
\frac{\theta^2}{13}$. Inductively choose $\{F_m\}_{m\inn}$
successive subsets of the naturals and $\{c_i\}_{i\in F_m}$
non-negative reals, satisfying the following:\vskip3pt

\noindent (i) $F_m\in\Sn$ and $z_m = \sum_{i\in F_m}c_ix_i, w_m =
\sum_{i\in F_m}c_iy_i$ are both $(n,\e_m)$ s.c.c. for all
$m\inn$.\vskip3pt

\noindent (ii) If we set
\begin{eqnarray*}
M_m &=& \max\{\max\supp z_m, \max\supp w_m\}\\
N_m &=& \min\{\min\supp z_{m}, \min\supp w_{m}\}
\end{eqnarray*}
then $M_m\big(\frac{1}{N_{m+1}} + 6\e_{m+1}\big) <
\frac{\de_{m+1}}{4}$, for all $m\inn$.\vskip3pt

We will show that for any $\{r_m\}_{m=1}^d\subset\mathbb{R}$, we
have that $\|\sum_{m=1}^d r_m w_m\| > \frac{\theta^2}{3}
\|\sum_{m=1}^d c_m z_m\|$.

Let $f\in W$. As always may assume that $1\geqslant r_m\geqslant
0, e_j^*(z_m)\geqslant 0, e_j^*(w_m)\geqslant 0, f(e_j)\geqslant
0$, for all $m,j\inn$. We may also assume that
$1\geqslant\|\sum_{m=1}^n r_k z_m\|>\theta$, therefore we may
assume that $1\geqslant f(\sum_{m=1}^d c_m z_m) > \theta$. By
Lemmas \ref{prop10}, \ref{lemqqq11} and \ref{lemqqq12}, there
exist $g^1, g^2, g^3\in W$, such that $(g^1 + g^2 +
g^3)(\sum_{m=1}^d r_m w_m)
> 2\theta f(\sum_{m=1}^d r_m z_m) -
13\sum_{m=1}^\infty\de_m
> 2\theta^2 - \theta^2 = \theta^2$. Hence $\|\sum_{m=1}^d r_m
w_m\|
>\frac{\theta^2}{3}$ and this means that $\|\sum_{m=1}^d r_m w_m\| > \frac{\theta^2}{3}
\|\sum_{m=1}^d r_m z_m\|$.

By symmetricity of the arguments it follows that $\{z_m\}_{m\inn}$
also dominates $\{w_m\}_{m\inn}$ , therefore $\{z_m\}_{m\inn}$ and
$\{w_m\}_{m\inn}$ are equivalent.

\end{proof}

\begin{cor}
The space $\X$ is quasi-minimal.\label{cor12}
\end{cor}

\begin{proof}
If $X, Y$ are block subspaces of $\X$, choose $\{x_k\}_{k\inn}$ in
$X$ and $\{y_k\}_{k\inn}$ in $Y$ normalized block sequences both
generating $\ell_1^n$ spreading models. Then obviously one may
pass to subsequences satisfying the assumption of Proposition
\ref{cor11}, therefore $X,Y$ contain further subspaces that are
isomorphic. Since any subspace contains an isomorph of a block
subspace, the result follows.
\end{proof}

\section{Strictly singular operators}
In this section we provide necessary and sufficient conditions for
a bounded operator defined on a subspace of $\X$, to be non
strictly singular. The proof of this is based on results from the
previous section and yields the following. For any $Y$ subspace of
$\X$ and $S_1,S_2,\ldots, S_{n+1}$ strictly singular operators on
$Y$, the composition $S_1S_2\cdots S_{n+1}$ is a compact operator.
We show that the strictly singular operators on the subspaces of
$\X$ admit non-trivial hyperinvariant subspaces. Next, we provide
a method for constructing strictly singular operators on subspaces
of $\X$, which is used to prove the non-separability of
$\mathcal{S}(Y)$ and also to build $S_1,\ldots,S_n$ in
$\mathcal{S}(Y)$, such that the composition $S_1\cdots S_n$ is
non-compact. We close this section by combining the above results with the properties of the $\al$-indices to show that $\{\mathcal{SS}_k(Y)\}_{k=1}^n$ is a strictly increasing family of two sided ideals.

\begin{thm} Let $Y$ be an infinite dimensional closed subspace of
$\X$ and $T:Y\rightarrow \X$ be a bounded linear operator. Then
the following assertions are equivalent.

\begin{itemize}

\item[(i)] $T$ is not strictly singular.

\item[(ii)] There exists a sequence $\{x_m\}_{m\inn}$ in $Y$
generating a $c_0$ spreading model, such that $\{Tx_m\}_{m\inn}$
generates a $c_0$ spreading model.

\item[(iii)] There exists $1\leqslant k \leqslant n$ and a
sequence $\{x_m\}_{m\inn}$ in $Y$, such that both
$\{x_m\}_{m\inn}$ and $\{Tx_m\}_{m\inn}$ generate an $\ell_1^k$
spreading model but no subsequences of $\{x_m\}_{m\inn}$ and
$\{Tx_m\}_{m\inn}$ generate an $\ell_1^{k+1}$ one.

\end{itemize}\label{prop13}

\end{thm}

\begin{proof}

Assume that there exists $1\leqslant k \leqslant n$ and a sequence
$\{x_m\}_{m\inn}$ in $Y$, such that both $\{x_m\}_{m\inn}$ and
$\{Tx_m\}_{m\inn}$ generate an $\ell_1^k$ spreading model but no
subsequences of $\{x_m\}_{m\inn}$ and $\{Tx_m\}_{m\inn}$ generate
an $\ell_1^{k+1}$ one.

If $\{x_m\}_{m\inn}$ converges weakly to a non-zero element $x$,
then $\{x_m - x\}_{m\inn}$, as well as $\{Tx_m - Tx\}_{m\inn}$
generate $\ell_1^k$ spreading models and no subsequences of them
generate an $\ell_1^{k+1}$ one. Therefore we may assume that they
are both normalized block sequences. Set $I_m = \ran(\ran
x_m\cup\ran Tx_m)$ and passing, if necessary, to a subsequence of
$\{x_m\}_{m\inn}$, $\{I_m\}_{m\inn}$ are increasing subsets of the
naturals.

Corollary \ref{cormmm9ell1} yields that
$\al_{n-k}\big(\{x_m\}_m\big)>0, \al_{n-k}\big(\{Tx_m\}_m\big)>0$
as well as $\al_{k^\prime}\big(\{x_m\}_m\big)=0,
\al_{k^\prime}\big(\{Tx_m\}_m\big)=0$, for $k^\prime < n-k$.

Choose $\{F_m\}_{m\inn}$ increasing subsets of the naturals
$\{c_i\}_{i\in F_m}$ non negative reals for all $m\inn$ such that
the following are satisfied.
\begin{enumerate}

\item[(i)] $\sum_{i\in F_m}c_ix_i$ as well as $\sum_{i\in
F_m}c_iTx_i$ are $(k,\e_m)$ s.c.c. with $\lim_m\e_m = 0$.

\item[(ii)] $F_m\in\mathcal{S}_k$

\end{enumerate}
Since $F_m\in\mathcal{S}_k$ and $\{x_m\}_{m\inn}$,
$\{Tx_m\}_{m\inn}$ generate $\ell_1^k$ spreading models, we
conclude that, if $z_m = \sum_{i\in F_m}c_ix_i$ for all $m\inn$,
then $\{z_m\}_{m\inn}$, as well as $\{Tz_m\}_{m\inn}$ are
seminormalized. Moreover, since
$\al_{k^\prime}\big(\{x_m\}_m\big)=0,
\al_{k^\prime}\big(\{Tx_m\}_m\big)=0$, for $k^\prime < n-k$, by
Proposition \ref{prop3} (ii) we conclude that
$\al_{n-1}\big(\{z_m\}_m\big)=0$ as well as
$\al_{n-1}\big(\{Tz_m\}_m\big)=0$. By Proposition \ref{cosm} we
conclude that passing, if necessary to a subsequence, both
$\{z_m\}_{m\inn}$ and $\{Tz_m\}_{m\inn}$ generate $c_0$ spreading
models.

Assume now that there exists a sequence $\{x_m\}_{m\inn}$ in $Y$
generating a $c_0$ spreading model, such that $\{T(x_m)\}_{m\inn}$
generates a $c_0$ spreading model. This means that
$\{x_m\}_{m\inn}$, as well as $\{Tx_m\}_{m\inn}$ are weakly null,
we may therefore assume that they are both normalized block
sequences. Apply Proposition \ref{propmmm10} and find
$\{F_m\}_{m\inn}$ increasing subsets of the naturals, such that if
$y_m = \sum_{i\in F_m}y_i$, then both $\{y_m\}_{m\inn}$ and
$\{Ty_m)\}_{m\inn}$ generate $\ell_1^n$ spreading models. Set $I_m
= \ran(\ran y_m\cup\ran Ty_m)$ and passing, if necessary, to a
subsequence of $\{y_m\}_{m\inn}$, $\{I_m\}_{m\inn}$ are increasing
subsets of the naturals. This means that the assumption of
Proposition \ref{cor11} is satisfied. Hence, there exists a
further block sequence $\{w_m\}_{m\inn}$ of $\{y_m\}_{m\inn}$,
such that $\{w_m\}_{m\inn}$ is equivalent to $\{Tw_m\}_{m\inn}$.
We conclude that $T$ is not strictly singular.

Assume now, that $T$ is not strictly singular and let $1\leqslant
k\leqslant n$. Then there exists $Z$ a subspace of $Y$, such that
$T|_Z$ is an isomorphism. Proposition \ref{propmmm14} yields that
any subspace of $\X$ contains a sequence generating an $\ell_1^k$
spreading model, such that no subsequence of it generates an
$\ell_1^{k+1}$ one, thus so does $Z$. Since $T|_Z$ is an
isomorphism, the third assertion must be true.

\end{proof}

The following definition is from \cite{ADST}

\begin{dfn}
Let $X$ be a Banach space and $k$ be a natural number. We denote by $\mathcal{SS}_k(X)$ the set of all bounded linear operators $T:X\rightarrow X$ satisfying the following: for every Schauder basic sequence $\{x_i\}_i$ in $X$ and $\e>0$, there exists $F\in\mathcal{S}_k$ and a vector $x$ in the linear span of $\{x_i\}_{i\in F}$ such that $\|Tx\| < \e\|x\|$.\label{defssk}
\end{dfn}

\begin{prp}
Let $Y$ be an infinite dimensional closed subspace of $\X$, $T:Y\rightarrow Y$ be a bounded linear operator and $1\leqslant k \leqslant n$. The following assertions are equivalent.
\begin{itemize}

\item[(i)] The operator $T$ is in $\mathcal{SS}_k(Y)$.

\item[(ii)] For every seminormalized weakly null sequence $\{x_i\}_i$ in $Y$, $\{Tx_i\}_i$ does not admit an $\ell_1^k$ spreading model.

\end{itemize}\label{propsskchar}
\end{prp}

\begin{proof}
The implication (i)$\Rightarrow$(ii) follows easily using Remark \ref{remarklifting} and therefore we omit it. Let us assume that (ii) holds, and towards a contradiction suppose that $T$ is not in $\mathcal{SS}_k(Y)$, i.e. there exist a normalized weakly null sequence $\{x_i\}_i$ in $Y$ and $\e>0$ satisfying the following: for every $F\in\mathcal{S}_k$ and real numbers $\{c_i\}_{i\in F}$ we have that
\begin{equation}
\|T(\sum_{i\in F}c_ix_i)\| \geqslant \e\|\sum_{i\in F}c_ix_i\|.\label{propsskchareq1}
\end{equation}

Let us first notice that $T$ is strictly singular. Indeed, if not then there exists a closed infinite dimensional subspace $Z$ of $Y$ such that $T|_Z$ is an isomorphism. Proposition \ref{propmmm14} yields that there exists a normalized weakly null sequence $\{z_i\}_i$ in $Z$ generating an $\ell_1^k$ spreading model. Since $T|_Z$ is an isomorphism, $\{Tz_i\}_i$ generates an $\ell_1^k$ spreading model as well, which contradicts (ii).

We shall now show that $\{Tx_i\}_i$ does not admit a $c_0$ spreading model. Assume that this is not the case, pass to a subsequence of $\{x_i\}_i$ and relabel so that $\{Tx_i\}_i$ generates a $c_0$ spreading model. Applying Theorem \ref{prop13} and Corollary \ref{cormmm9dich}, we may assume that $\{x_i\}_i$ generates an $\ell_1$ spreading model. This implies that there exists $F\in\mathcal{S}_1$ such that $\|T(\frac{1}{\#F}\sum_{i\in F}x_i)\| < \e\|\frac{1}{\#F}\sum_{i\in F}x_i\|$, which contradicts \eqref{propsskchareq1}.

Corollary \ref{cormmm9dich} and Remark \ref{remarklifting} imply that there exist natural numbers $1\leqslant d\leqslant m\leqslant n$ and a subsequence of $\{x_i\}_i$, again denoted by $\{x_i\}_i$, such that $\{Tx_i\}_i$ generates an $\ell_1^d$ spreading model and does not admit an $\ell_1^{d+1}$ one, while $\{x_i\}_i$ generates an $\ell_1^m$ spreading model and does not admit an $\ell_1^{m+1}$ one. Theorem \ref{prop13} implies that $d+1\leqslant m$. Combining the above it is easy to see that there exists $F\in\mathcal{S}_{d+1}$ and real numbers $\{c_i\}_{i\in F}$ such that $\|T(\sum_{i\in F}c_ix_i)\| < \e\|\sum_{i\in F}c_ix_i\|$. However, (ii) yields that $d+1\leqslant k$ and hence $F\in\mathcal{S}_k$ which contradicts \eqref{propsskchareq1}.
\end{proof}

\begin{prp}
Let $Y$ be an infinite dimensional closed subspace of $\X$, and $\{x_m\}_{m\inn}$ be a
seminormalized weakly null sequence in $Y$. Then for every $1\leqslant k \leqslant n$ and
$S_1,S_2,\cdots S_k:Y\rightarrow Y$ strictly singular operators,
$\{S_1S_2\cdots S_kx_m\}_{m\inn}$ has no subsequence generating an $\ell_1^{n+1-k}$ spreading model. In particular $S_1S_2\cdots S_k$ is in $\mathcal{SS}_{n+1-k}(Y)$.\label{propqqq14}
\end{prp}

\begin{proof}
The second assertion of this proposition evidently follows from the first one and Proposition \ref{propsskchar}.
We prove the first assertion by induction on $k$. For $k = 1$ and
$S:Y\rightarrow Y$ a strictly singular operator, assume that $\{Sx_m\}_m$ generates an $\ell_1^n$ spreading model. The boundedness of $S$ yields that $\{x_m\}_m$ must also generate an $\ell_1^n$ spreading model, while by Corollary \ref{cormmm11} neither $\{x_m\}_m$ nor $\{Sx_m\}_m$ admit an $\ell_1^{n+1}$ spreading model. Theorem \ref{prop13} yields that $S$ cannot be strictly singular which is absurd.

Assume now that the statement holds for some $1\leqslant k <n$ and let $S_1,\ldots,S_{k+1}:Y\rightarrow Y$ be strictly singular operators. If $\{S_1S_2\cdots S_{k+1}x_m\}_m$ generates an $\ell_1^{n-k}$ spreading model, then the boundedness of the operators yields that $\{S_2\cdots S_{k+1}x_m\}_m$ generates an $\ell_1^{n-k}$ spreading model as well. By the inductive assumption it follows that neither of the sequences $\{S_1S_2\cdots S_{k+1}x_m\}_m$, $\{S_2\cdots S_{k+1}x_m\}_m$ admits an $\ell_1^{n+1-k}$ spreading model. Once more, Theorem \ref{prop13} yields that $S_1$ cannot be strictly singular, a contradiction which completes the proof.

\end{proof}

\begin{prp}Let $Y$ be an infinite dimensional closed subspace of
$\X$ and $S_1,S_2,\ldots,S_{n+1}:Y\rightarrow Y$ be strictly
singular operators. Then $S_1S_2\cdots S_{n+1}$ is
compact.\label{cor14}
\end{prp}

\begin{proof}
Since $\X$ is reflexive, it is enough to show that for any weakly
null sequence $\{x_m\}_{m\inn}$, we have that $\{S_1S_2\cdots
S_{n+1}x_m\}_{m\inn}$ norm converges to zero. By Proposition \ref{propqqq14}, the sequence $\{S_2\cdots S_{n+1}x_m\}_m$ does not admit an $\ell_1$ spreading model and hence, by Corollary \ref{cormmm9dich} it is either norm null or it has some subsequence generating a $c_0$ spreading model.

If it is norm null, then there is nothing to prove. If, on the other hand, $\{S_2\cdots S_{n+1}x_m\}_m$ generates a $c_0$ spreading model, then Theorem \ref{prop13} and the fact that $S_1$ is strictly singular yield that $\{S_1S_2\cdots
S_{n+1}x_m\}_{m\inn}$ norm converges to zero.

\end{proof}

\begin{cor}Let $Y$ be an infinite dimensional closed subspace of
$\X$ and $S:Y\rightarrow Y$ be a non zero strictly singular
operator. Then $S$ has a non-trivial closed hyperinvariant
subspace.\label{ssisp}
\end{cor}

\begin{proof}
Assume first that $S^{n+1} = 0$. Then it is straightforward to
check that $\ker S$ is a non-trivial closed hyperinvariant
subspace of $S$.

Otherwise, if $S^{n+1} \neq 0$, then Cor. \ref{cor14} yields that
$S^{n+1}$ is compact and non zero. Since $S$ commutes with
$S^{n+1}$, by Theorem 2.1 from \cite{S}, it is enough to check
that for any $\al,\be\in\mathbb{R}$ such that $\be\neq 0$, we have
that $(\al I - S)^2 + \be^2I \neq 0$ (see also \cite[Theorem 2]{H}). Since $S$ is strictly
singular, it is easy to see that this condition is satisfied.
\end{proof}

\begin{rmk} The space $\X$ is also defined over the complex field,
satisfying all the above and following properties. For the complex
$\X$ the above Corollary is an immediate consequence of the
classical Lomonosov Theorem \cite{L}.
\end{rmk}

\begin{rmk} A well known result due to M. Aronszajn and K. T. Smith \cite{AS},
asserts that compact operators always admit non-trivial invariant
subspaces. As it is shown by C. J. Read in \cite{R}, there do
exist strictly singular operators on Banach spaces, not admitting
any non-trivial invariant subspaces. Therefore, one may not hope
to extend M. Aronszajn's and K. T. Smith's result to strictly
singular operators. In \cite{AH} a hereditarily indecomposable
Banach space $\mathfrak{X}_{K}$ is presented satisfying the scalar
plus compact property. It follows that any operator acting on this
space, admits a non-trivial closed invariant subspace. Moreover,
in \cite{AM} a reflexive hereditarily indecomposable Banach space
$\mathfrak{X}_{ISP}$ is constructed such that any operator acting
on a subspace of $\mathfrak{X}_{ISP}$, admits a non-trivial closed
invariant subspace.

\end{rmk}

The next Corollary is an immediate consequence of the previous
one.

\begin{cor}Let $Y$ be an infinite dimensional closed subspace of
$\X$ and $T:Y\rightarrow Y$ be a linear operator that commutes
with a non zero strictly singular operator. Then $T$ admits a
non-trivial closed invariant subspace.\label{corx20}
\end{cor}

Before stating the next theorem we need the following lemma
concerning sequences that do not have a subsequence generating an
$\ell^{k+1}_1$ spreading model.

\begin{lem} Let $0\leqslant k\leqslant n-1$ and $\{x_i\}_{i\inn} \subset Ba(\X)$ be a block
sequence such that no subsequence of it, generates an
$\ell_1^{k+1}$ spreading model. Then for every $m\inn$ there
exists $L\in[\mathbb{N}]^\infty$ such that for any $m\leqslant
F_1<\cdots< F_m$ maximal $\mathcal{S}_k$ subsets of $L$ the
following holds.
\begin{equation*}
\|\sum_{j=1}^m\sum_{i\in F_j}c_i^{F_j}x_i\| < 2
\end{equation*}\label{fixedm}
\end{lem}

\begin{proof}
Fix $m \inn$  and let  $\mathcal{G}$ to be the collection of
finite sets $F$ satisfying $F = \cup_{j=1}^m F_j$, where $m \leq
F_1 < \ldots < F_m$ are maximal $\Sk$ sets for all $i \in \{1,
\ldots , m\}$ and

$$ \| \sum_{i=1}^m \sum_{i \in F_j} c_i^{F_j} x_i \|  \geqslant 2.$$

Assume the conclusion of the lemma is false. Then, by definition,
the collection $\mathcal{G}$ is large in the $\nn$. A theorem of
Nash-Williams \cite{Nash1} gives us an $L \in [\nn]$ such that
$\mathcal{G}$ for all $M \in [L]$ and initial segment of $M$ is in
$\mathcal{G}$ (i.e. $\mathcal{G}$ is very large in $L$).

Therefore for any $F_1 <\ldots <F_m$ (assume $\min L \geqslant m$) maximal $\Sk$ subsets of $L$ we have

\begin{equation}
\| \sum_{i=1}^m \sum_{i \in F_j} c_i^{F_j} x_i \|  \geqslant 2
\label{geq2}
\end{equation}

We show this yields a contradiction. Let $(F_j)_j$ be an
increasing sequence of maximal $\Sk$ subset of $L$ and define $y_j
= \sum_{i \in F_j} c_i^{F_j} x_i$.  By Proposition \ref{propmmm7},
$\al_{n-k-1}(\{x_i\}_i)=0$. Since each $\{y_j\}_j \subset Ba(\X)$
and each $y_j$ is a $(k,3/\min F_j)$ s.c.c Proposition
\ref{prop3}(2), implies that $\al_{n-1}(\{y_j\}_j) =0$.  By
Proposition \ref{cosm} there is a subsequence of $\{y'_j\}_{j
\inn}$ of $\{y_j\}_{j \inn}$ such that for $m \leqslant  k_1 < \cdots <
k_m$ we have
$$\| \sum_{j =1}^m \sum_{i \in F_{k_j}} c_i^{F_{k_j}} x_i \| = \| \sum_{j=1}^m y'_{k_j} \|  <2.$$
This contradicts (\ref{geq2}).
\end{proof}

The next proposition is an intermediate step towards showing that
for any $Y$ infinite dimensional closed subspace of $\X$, there
exist $S_1,\ldots,S_n:Y\rightarrow Y$ strictly singular operators,
such that $S_1\cdots S_n$ is non compact.

\begin{prp} Let $0 \leqslant k \leqslant n-1$ and $Y$ be an infinite dimensional closed subspace of
$\X$. Let also $\{x_i^*\}_{i\inn}$ be a sequence in
${\mathfrak{X}_{_{0,1}}^{n*}}$ generating a $c_0^{k+1}$ spreading
model and $\{x_i\}_{i\inn}$ be a seminormalized weakly null
sequence in $Y$, such that no subsequence of it generates an
$\ell_1^{k+1}$ spreading model. Then passing, if necessary, to
subsequences of $\{x_i^*\}_{i\inn}$ and $\{x_i\}_{i\inn}$, the map
$T:Y\rightarrow Y$ with $Tx = \sum_{1=1}^\infty x_{i}^*(x)x_i$ is
bounded, strictly singular and non-compact.\label{operator}
\end{prp}

\begin{proof}
Passing, if necessary, to a subsequence, we may assume that
$\{x_i\}_{i\inn}$ is a normalized block sequence.

It follows from Lemma \ref{fixedm} and a standard diagonal argument that there
is an $L \in [\nn]$  such for all $m \in \nn$ and  $m\leqslant F_1<\cdots< F_m$ maximal
$\mathcal{S}_k$ sets in $L$

\begin{equation}
\| \sum_{j=1}^m \sum_{i \in F_j} c^{F_j}_i x_i \| < 2.
\label{Q3eq1}
\end{equation}

Choose a subsequence $(i_j)_{j \in \nn}$ of $\nn$
such that $i_j\geqslant 2^{j+3}+1$ for all
$j\inn$. We claim that the map
$$Tx = \sum_{j\in L} x_{i_j}^*(x)x_{j}$$
is the desired one.

Let $x\in Y, \|x\| = 1$ and $x^*\in Y^*, \|x^*\| = 1$. We may
assume that $x^*(x_j) \geqslant 0$ for all $j\in L$. We partition
$L$ in the following way: For $q=0,1,\ldots$ set
\begin{eqnarray*}
B_q &=& \{j \in L : \frac{1}{2^{q+1}} < x^*(x_j) \leqslant
\frac{1}{2^q}\}\\
C_q &=& \{j\in B_q: j\geqslant q+1\}\\
D_q &=& \{j \in B_q: j\leqslant q\}
\end{eqnarray*}

Evidently we have
\begin{equation}
|\sum_{j\in D_q}x^*(x_j)x^*_{i_j}(x)| \leqslant \frac{q}{2^q}
\label{Q3eq2}
\end{equation}
Decompose $C_q$ into successive subsets
$\{C_q^\ell\}_{\ell=0}^{p_q}$ of $L$ such that the following are
satisfied:
\begin{itemize}

\item[(i)] $C_q = \cup_{\ell = 0}^{p_q}C_q^\ell$

\item[(ii)] $C_q^0 = C_q\cap\{q+1,\ldots,2^{q+1}\}$ and for
$\ell>0$ $C_q^\ell$ is a maximal $\mathcal{S}_k$ set (except
perhaps the last one).

\end{itemize}
We claim that $p_q< 2^{q+3}$.  Let $I_q \subset \{1, \ldots
,p_q\}$ be an $\mathcal{S}_1$ set such that $\# I_q \geqslant p_q/2$.
From \eqref{Q3eq1} and the definition of $B_q$ we have
\begin{equation*}
\begin{split}
2 &> \|\sum_{\ell \in I_q}\sum_{j \in C_q^\ell} c_j^{C_q^\ell}
x_{j} \| \geqslant \sum_{\ell \in I_q}\sum_{j \in
C_q^\ell}c_j^{C_q^\ell} x^*(x_{j}) \geqslant \frac{p_q}{2^{q+2}}
\end{split}
\end{equation*}
Therefore $p_q < 2^{q+3}$.

Now set
\begin{equation*}
G_q^\ell = \{i_j: j\in
C_q^\ell\}\quad\text{for}\;\ell=0,\ldots,p_q.
\end{equation*}
Then it is easy to check the following.
\begin{itemize}

\item[(i)] $G_q^0\in\mathcal{S}_1$ and $\min G_q^0 >
2^{q+3}$

\item[(ii)] $G_q^\ell\in\mathcal{S}_k$ for $\ell>0$.

\end{itemize}
Since $p_q< 2^{q+3}$, the set $G_q =
\cup_{\ell = 0}^{p_q}G_q^\ell \in\mathcal{S}_{k+1}$.
Since $\{x_i^*\}_{i\inn}$ generates a $c_0^{k+1}$
spreading model, we conclude the following:
\begin{equation}
|\sum_{j\in C_q}x^*(x_j)x^*_{i_j}(x)| < 2\max\{|x^*(x_j)|: j\in
C_q\}\label{Q3eq3}
\end{equation}

Summing up \eqref{Q3eq2} and \eqref{Q3eq3} we conclude that
$\|T\|\leqslant 2\sum_{q = 0}^\infty\frac{1+q}{2^q}$.

To see that $T$ is non-compact consider the biorthogonal functionals
$\{f_k\}_{k \in L}$ of $\{x_{i_j}^*\}_{j \in L}$.  Since $\{f_k\}_{k \in L}$ is
a seminormalized sequence we have
$$\| T(f_k - f_m)\| = \| x_k - x_m\|$$
for $m \not= k$ in $L$.  Therefore $\{T(f_k)\}_{k \inn}$ has no norm convergent
subsequence.

To prove that $S$ is strictly singular, first notice that for
$x\in Y, \|x\| = 1, x^*\in Y^*, \|x^*\| = 1, j_0\inn$, we have
that

\begin{eqnarray*}
x^*\big(Tx\big) &\leqslant& \sum_{q=0}^{q_0-1}|\sum_{j\in B_q}
x^*_{i_j}(x)x^*(x_j)| + 2\sum_{q=q_0}^\infty\frac{(q+1)}{2^q}\\
&\leqslant& \sum_{q=0}^{q_0-1}\Big(|\sum_{j\in D_q}
x^*_{i_j}(x)x^*(x_j)| + |\sum_{j\in C_q}
x^*_{i_j}(x)x^*(x_j)|\Big) +
2\sum_{q=q_0}^\infty\frac{(q+1)}{2^q}\\
&<&\sum_{q=0}^{q_0-1}(q+2)\sup\{|x^*_{i_j}(x)|: j\inn\} +
2\sum_{q=q_0}^\infty\frac{(q+1)}{2^q}
\end{eqnarray*}

Therefore $\displaystyle{\|Tx\| \leqslant \frac{q_0^2 +
3q_0}{2}\sup\{|x^*_{i_j}(x)|: j\inn\} +
2\sum_{q=q_0}^\infty\frac{(q+1)}{2^q}}$

Let $Z$ be an infinite dimensional closed subspace of $Y$ and
$\e>0$. Since $Z$ does not contain $c_0$, it follows that for any
$\delta>0$ there exists $x\in Z, \|x\| = 1$, such that
$\sup\{|x^*_{i_j}(x)|: j\inn\}<\delta$. For appropriate choices of
$q_0$ and $\delta$, it follows that there exists $x\in X, \|x\| =
1$ such that $\|Tx\| < \e$, thus $T$ is strictly singular.

The proof of the boundedness is based on the proof of Proposition
3.1 from \cite{ADT} and the proof of the strict singularity of $T$
originated from an unpublished result due to A. Pelczar-Barwacz.

\end{proof}

\begin{rmk} The proof of the above proposition actually yields,
that for $L,M$ infinite subsets of the naturals, the map $T_{L,M}
= \sum_{i=1}^\infty x^*_{L(i)}(x)x_{M(i)}$ remains bounded,
strictly singular and non compact.\label{remarkoperator}
\end{rmk}

\begin{cor} For any  infinite dimensional closed subspace $Y$ of
$\X$, the ideal $\mathcal{S}(Y)$ of strictly singular operators is
non-separable.\label{cornonseparability}
\end{cor}

\begin{proof}
Choose $\{x_i^*\}_{i\inn}$ a seminormalized sequence in
${\mathfrak{X}_{_{0,1}}^{n*}}$ generating a $c_0^n$ spreading
model and $\{x_i\}_{i\inn}$ a seminormalized weakly null sequence
in $Y$ not having a subsequence generating an $\ell_1^n$ spreading
model, such that the map $T:Y\rightarrow Y$ with $Tx =
\sum_{i=1}^\infty x_i^*(x)x_i$ is bounded, strictly singular and
non-compact. By the Remark \ref{remarkoperator}, for any $L$
infinite subset of the naturals, the operator $T_L:Y\rightarrow Y$
with $T_Lx = \sum_{i=1}^\infty x_L(i)^*(x)x_i$ is bounded,
strictly singular and non-compact. Therefore $\mathcal{S}(Y)$
contains an uncountable $\e$-separated subset, hence it is
non-separable.

\end{proof}

\begin{prp} Let $Y$ be an infinite dimensional closed subspace of
$\X$. Then there exist $S_1,\ldots,S_n:Y\rightarrow Y$ strictly
singular operators, such that for $0\leqslant k \leqslant n-2$ the composition $S_{n-k}\cdots S_n$
is in $\mathcal{SS}_{n-k}(Y)$ and not in $\mathcal{SS}_{n-k-1}(Y)$ and $S_1\cdots S_n$ is in $\mathcal{SS}_1(Y)$ and it is not compact.\label{counter}
\end{prp}

\begin{proof}
Using Proposition \ref{propmmm14}, Remark \ref{strongell1},
Proposition \ref{operator} and Remark \ref{remarkoperator}, for
$k=1,\ldots,n$ choose $\{x_{k,i}\}_{i\inn}$ normalized weakly null
sequences in $Y$ and $\{x_{k,i}^*\}_{i\inn}$ normalized weakly
null sequences in ${\mathfrak{X}_{_{0,1}}^{n*}}$ satisfying the
following.

\begin{enumerate}

\item[(i)] $\{x_{k,i}\}_{i\inn}$ generates an $\ell_1^{k-1}$
spreading model and no subsequence of it generates an $\ell_1^{k}$
one for $k=2,\ldots,n$, while $\{x_{1,i}\}_{i\inn}$ generates a
$c_0$ spreading model.

\item[(ii)] $\{x_{k,i}^*\}_{i\inn}$ generates a $c_0^k$ spreading
model for $k=1,\ldots,n$.

\item[(iii)] There exists $\e_k>0$ such that $x^*_{k+1,i}(x_{k,i})
> \e_k$ for all $i\inn$ and $x^*_{k+1,i}(x_{k,j}) = 0$ for $i\neq j$, $k=1,\ldots,n-1$.

\item[(iv)] The map $S_k:Y\rightarrow Y$ with $S_k(x) =
\sum_{i=1}^\infty x_{k,i}^*(x)x_{k,i}$ is bounded strictly
singular and non-compact.

\end{enumerate}

We shall inductively prove the following. For $k=0,\ldots,n-1$
there exists  a sequence of seminormalized
positive real numbers $\{c_{k,i}\}_{i\inn}$ such
that

\begin{equation*}
S_{n-k}\cdots S_{n-1}S_nx = \sum_{i=1}^\infty
c_{k,i}x^*_{n,i}(x)x_{n-k,i}
\end{equation*}

For $k = 0$, the assumption holds, for $c_{0,i} = 1$ for all
$i\inn$. Assume that it holds for some
$k< n-1$. Then, by the inductive assumption

\begin{eqnarray*}
S_{n-k-1}\cdots S_nx &=& \sum_{i=1}^\infty
x_{n-k-1,i}^*(\sum_{j=1}^\infty
c_{k,j}x^*_{n,j}(x)x_{n-k,j})x_{k,i}\\
&=& \sum_{i=1}^\infty
c_{k,i}x_{n-k-1,i}^*(x_{n-k,i})x^*_{n,i}(x)x_{k,i}
\end{eqnarray*}

Set $c_{k+1,i} = c_{k,i}x_{n-k-1,i}^*(x_{n-k,i})$ for all $i\inn$.
Then $c_{k+1,i} > c_{k,i}\e_{n-k}$, for all $i\inn$, therefore
$\{c_{k+1,i}\}_{i\inn}$ is seminormalized. The induction is complete.

Let now $0\leqslant k \leqslant n-2$. Proposition \ref{propqqq14} yields that $S_{n-k}\cdots S_n$ is in $\mathcal{SS}_{n-k}(Y)$. Moreover, if we consider $\{y_i\}_i$ to be a seminormalized sequence in $Y$, biorthogonal to $\{x^*_{n,i}\}_{i\inn}$, then $S_{n-k}\cdots S_ny_i = c_{k,i}x_{n-k,i}$ and therefore by (i) $\{S_ny_i\}_i$ generates an $\ell_1^{n-k-1}$ spreading model. Proposition \ref{propsskchar} yields that $S_{n-k}\cdots S_n$ is not in $\mathcal{SS}_{n-k-1}(Y)$

The fact that $S_1\cdots S_n$ is in $\mathcal{SS}_1(Y)$ and it is not compact is proved similarly.

\end{proof}

\begin{prp}
Let $Y$ be an infinite dimensional closed subspace of $\X$. Then $\mathcal{K}(Y)\subsetneq \mathcal{SS}_1(Y)\subsetneq \mathcal{SS}_2(Y)\subsetneq\cdots\subsetneq\mathcal{SS}_n(Y) = \mathcal{S}(Y)$ and for every $1\leqslant k\leqslant n$, $\mathcal{SS}_k(Y)$ is a two sided ideal.\label{propsskfinal}
\end{prp}

\begin{proof}
The fact that $\mathcal{SS}_n(Y) = \mathcal{S}(Y)$ follows from Proposition \ref{propqqq14} while the fact that $\mathcal{K}(Y)\subsetneq \mathcal{SS}_1(Y)\subsetneq \mathcal{SS}_2(Y)\subsetneq\cdots\subsetneq\mathcal{SS}_n(Y)$ follows from Proposition \ref{counter}. Fix $1\leqslant k \leqslant n$. We will show that $\mathcal{SS}_k(Y)$ is a two sided ideal  and for that it is enough to show that whenever $S, T$ are in $\mathcal{SS}_k(Y)$, then so is $S + T$. The other properties of an ideal were verified in \cite{ADST} and hold for any space.  

We shall show that for every seminormalized weakly null  sequence $\{x_i\}_i$ in $Y$, $\{(S+T)x_i\}_i$ does not admit an $\ell_1^k$ spreading model and by Proposition \ref{propsskchar} we will be done.

We may assume that $\{Sx_i\}_i, \{Tx_i\}_i$ and $\{(S+T)x_i\}_i$ are all seminormalized block sequences. Since $S$ and $T$ are both in $\mathcal{SS}_k(Y)$, by Proposition \ref{propsskchar}  neither $\{Sx_i\}_i$ nor $\{Tx_i\}_i$ admits an $\ell_1^k$ spreading model. Proposition \ref{propmmm7} yields that $\al_{k^\prime}\big(\{Sx_i\}_i\big) = 0$ as well as $\al_{k^\prime}\big(\{Tx_i\}_i\big) = 0$ for $k^\prime < n-k+1$. It immediately follows from the definition of the $\al$-index that $\al_{k^\prime}\big(\{(S+T)x_i\}_i\big) = 0$ for $k^\prime < n-k+1$. Once more, Proposition \ref{propmmm7} yields that $\{(S+T)x_i\}_i$ does not admit an $\ell_1^k$ spreading model.

\end{proof}

\section*{The space $\mathfrak{X}_{0,1}^\omega$}

Recall that
$$\mathcal{S}_\omega=\{ F \subset \nn : n \leq F \mbox{ and } F \in \mathcal{S}_n \mbox{ for some }n \inn\}.$$
The space $\Xomega$ is defined in the natural way allowing $\mathcal{S}_\omega$-admissible successive subsets of $\nn$.
In this section let $W$
denote the norming set of $\Xomega$.
For this space we have the following proposition.

\begin{prp}
The following hold for $\Xomega$.
\begin{itemize}
\item[(i)] Every normalized weakly null sequence has a subsequence generating a $c_0$ or $\ell^\omega_1$ spreading model.
\item[(ii)] Every non-trivial spreading model of $\Xomega$ is either isomorphic to $c_0$ or $\ell_1$.
\item[(iii)] Every subspace of $\Xomega$ admits a spreading model isometric to $c_0$ and a spreading model isometric to $\ell_1$.
\item[(iv)] Let $Y$ be an infinite dimensional subspace of $\Xomega$. The following are equivalent.
\begin{itemize}
\item[(a)] $T:Y\to Y$ is a strictly singular
\item[(b)] There is a weakly null sequence $\{x_i\}_{i \inn}$ such that both $\{x_i\}_{i \inn}$ and  $\{T x_i\}_{i \inn}$ generate a $\ell_1^\omega$ spreading model
\item[(c)] There is a
weakly null sequence $\{y_i\}_{i \inn}$ such that both $\{y_i\}_{i \inn}$ and  $\{T y_i\}_{i \inn}$ generate a $c_0$ spreading model.
\end{itemize}
\end{itemize}
\label{omegaprop}
\end{prp}

Since the proof of (ii) and (iv) are almost identical to the finite order case, we omit them. Below we include the sketches of the proofs
of (i) and (iii). These are also similar to the
corresponding proofs for $\X$, however, there are some technical differences that are worth pointing out.

Clearly for each $1 \leq \xi <\omega_1$ the space $\Xxi$ can be defined using the Schreier family
$\mathcal{S}_\xi$ where appropriate. See \cite{AA} for the definition of $\mathcal{S}_\xi$. Whenever $\xi$ is a countable limit ordinal satisfying $\eta + \xi = \xi$ for all $\eta<\xi$, we claim that the above proposition
holds replacing $\omega$ with $\xi$. If $\xi$ is of the form $\xi = \zeta +(n-1)$, where $\zeta$ is a limit ordinal satisfying the above condition
 and $n \inn$, we have observed that  the spreading models in
this space behave analogously to those in $\X$. The technical difficulty in including the proofs of these results is that they require us to introduce the higher order repeated averages and modify the proofs to accommodate
more complicated nature of the Schreier sets of transfinite order. However, there does not seem to be any non-technical obstruction to proceeding in this direction.

It is worth pointing out that for countable ordinal numbers $\xi$ failing the condition $\eta + \xi = \xi$ for all $\eta<\xi$, the space $\Xxi$ fails to satisfy (i). For example in the space $\mathfrak{X}_{_{^{0,1}}}^{\omega\cdot2}$ every seminormalized weakly null sequence admits either $c_0$ as a spreading model, or $\ell_1^\zeta$, for $\omega\leqslant\zeta\leqslant\omega\cdot2$.

The following definition is found in \cite[Definition 3.1]{AM}.

\begin{dfn}
Let $\{ x_k \}_{k \inn}$ be a block sequence in $\Xomega$.

We write
$\alpha_{<\omega}(\{x_i\}_{i \inn})=0$ if for any $n \inn$, any fast growing sequence $\{\alpha_q\}_{q \inn}$
of $\alpha$-averges in $W$ and for any $\{F_k\}_{k \inn}$ increasing sequence of subsets of $\nn$, such
that $\{\alpha_q\}_{q \in F_k}$ is $\mathcal{S}_{n}$, the following holds: For any subsequence $\{x_{n_k}\}_{k \inn}$
of $\{x_k\}_{k \inn}$ we have $\lim_{k}\sum_{q \in F_k} |\alpha_q (x_{n_k})|=0$. If this is not the case, we write $\alpha_{<\omega}(\{x_i\}_{i \inn})>0$. \label{omegaindex}
\end{dfn}

Notice that for any limit ordinal $\xi <\omega_1$
it is easy to define the corresponding index $\alpha_{<\xi}$ using the sequence or ordinals increasing up to $\xi$. The next proposition is proved in \cite[Proposition 3.3]{AM}. We note that in contrast with the finite order case, the argument is not completely trivial; however, for the sake of
brevity we omit it.

\begin{prp}
Let $\{x_k\}_{k \inn}$ be a block sequence in $\Xomega$. The following are equivalent.
\begin{itemize}
\item[(i)] $\alpha_{<\omega} (\{x_k\}) =0$
\item[(ii)] For any $\e >0$ there exists $j_0 \inn$ such that for any $j \geqslant j_0$ there is an $k_j \inn$ such that for
any $k \geqslant k_j$, and for any $\{\alpha_q\}_{q =1}^d$ $\mathcal{S}_{j}$-admissible and very fast growing sequence
of $\alpha$-averages such that $s(\alpha_q)>j_0$ for $q = 1, \ldots , d$, we have that $\sum_{q=1}^d |\alpha_q(x_k)| <\e$.
\end{itemize}
\label{higherorderindex}
\end{prp}

As in the finite case we need use the index to establish existence of the spreading models.

\begin{prp}
Let $\{ x_i \}_{i \inn}$ be normalized block sequence in $\Xomega$. Then the following hold:
\begin{itemize}
\item[(i)] If $\alpha_{<\omega}(\{x_i\}) >0$, then, by passing to a subsequence, $\{ x_i \}_{i \inn }$
generates a strong $\ell_1^\omega$ spreading model.
\item[(ii)] If $\alpha_{<\omega}(\{x_i\}) =0$ then there is a sequence of $\{x_i\}$ that generates a $c_0$ spreading model.
\end{itemize}
\end{prp}

\begin{proof}
First we prove (i). By Definition \ref{omegaindex} there is an $d \inn$, a very fast growing
sequence of $\alpha$-averages $\{\alpha_q\}_{q \inn}$ in $W$, and sequence $\{F_i\}_{i \inn}$
of successive finite subsets such that $\{\alpha_q\}_{q \in F_i}$ is $\mathcal{S}_{d}$ for each $i \inn$ and
 $$\sum_{q \in F_i} |\alpha_q (x_i)|>\e.$$
Relabeling so that $F_1 \geqslant d$ we have that  $(F_i)_{i \inn}$ that for $G \in \mathcal{S}_\omega
$, we have $\cup_{i \in G} F_i \in S_\omega$.  Pass to a further subsequence such that of $\{x_i\}_{i \inn}$ such that
$$\max \supp ( \sum_{q \in F_i} \alpha_q ) < \min \supp x_{i+1}.$$
Let $x^*_i = \sum_{q \in F_i} \alpha_q$.  Note that $\e < \|x_i^*\| \leqslant 1$. If $G \in S_\xi$ the above argument yields that $\sum_{i \in G} x^*_i $ is a Schreier functional.
Therefore $\| \sum_{i \in G} x^*_i \| \leq 1$. This implies $\{x_i^*\}_{i \inn}$ generates a $c_0^\xi$ spreading model, as desired.

The proof has the same structure as the proof of Proposition \ref{cosm} and so we will sketch some of the details.
Let $\{\e_i\}_{i \inn}$ be a summable sequence of positive reals such that $\e_i>3 \sum_{j>i} \e_j$
for all $i \inn$. Using Proposition \ref{higherorderindex}, inductively choose a subsequence, again denoted
by $\{ x_i\}_{i \inn}$, such that for $i_0 \geqslant 2$ and $j_0=\max\supp x_{i_0-1}$ if $\{\alpha_q\}_{q=1}^\ell$ is $\mathcal{S}_{j_0}$-
admissible $s(\alpha_q) \geqslant \min\supp x_{i_0}$ then for all $i \geq i_0$
$$\sum_{q=1}^\ell |\alpha_q (x_i)| < \frac{\e_{i_0}}{i_0 \max \supp x_{i_0 -1}}.$$
As before, we will show that for any $t \leqslant i_1 < \ldots < i_t$, $F \subset\{1, \ldots , t\}$ we have
$$|\alpha(\sum_{j \in F} x_{i_j})|< 1+ 2 \e_{i_{\min F}}.$$
whenever $\alpha$ is an $\alpha$-average and
$$|g(\sum_{j \in F} x_{i_j})|< 1+ 3 \e_{i_{\min F}}.$$
whenever $g$ is Schreier functional. This implies the proposition.

For functionals in $W_0$ the above is clearly true. Assume for $m \geqslant 0$ that above holds for
$t \leqslant i_1 < \ldots < i_t$ and any functional in $W_m$. In the first case, let  $t \leqslant i_1 < \ldots < i_t$
and $\alpha \in W_{m+1}$. In this case, we refer the reader to the analogous step in the proof of Proposition \ref{cosm}.

Let $g \in W_{m+1}$ such that $g= \sum_{q=1}^d \alpha_q$ be a Schreier functional. We
assume without loss of generality that
\begin{equation}
\ran g \cap \ran x_{i_j} \neq \varnothing \mbox{ for all } j = 1, \ldots t.
\label{wlog}
\end{equation}
Set
\begin{equation*}
q_0 = \min\{q: \max \supp \alpha_q \geqslant \min\supp x_{i_{2}}\}.
\end{equation*}
By definition of $\mathcal{S}_\omega$, $\{ \alpha_q\}_{q=1}^d$ is $S_{\min\supp \alpha_1}$-admissible.
Also, by definition, for $q>q_0$
$$s(\alpha_q) > \max\supp \alpha_{q_0} \geq \min \supp x_{i_2}.$$
Using (\ref{wlog})
$$ \min\supp \alpha_1 \leqslant \max \supp x_{i_1}.$$
These facts together allow us to use or initial assumption on the sequence $\{x_i\}_{i \inn}$ (for $i_0=i_2$) and conclude
that for $j \geqslant 2$
 \begin{equation}
 \sum_{q>q_0} |\alpha_q (x_{i_j})| < \frac{\e_{i_2}}{i_2 \max\supp x_{i_1}}.
 \end{equation}
Using the fact that $i_2 \geqslant t$, it follows that
$$\sum_{q>q_0} |\alpha_q (\sum_{j=1}^t x_{i_j})| < \e_{i_{1}}.$$
As before we consider two more cases.

{\em Case 1:} Assume that for $q< q_0$, $\alpha_q( \sum_{j=1}^t x_{i_j})=0$. In this case apply the induction
for $\alpha_{q_0}.$

{\em Case 2:} Alternatively, assume $s(\alpha_{q_0}) \geqslant \min \supp x_{i_{1}}$. In this case, since the singleton $\alpha_{q_0}$ is $\mathcal{S}_0$ admissible, we can apply our initial assume to conclude that $|\alpha_{q_0}( \sum_{j=1}^t x_{i_j})| < \e_{j_1}.$
Combining previous estimates gives the desired result.

 \end{proof}

 The next proposition implies item (iii) of Proposition \ref{omegaprop}

 \begin{prp}
 Let $\{x_k\}_{k \inn}$ be a normalized block sequence in $\Xomega$ and $\{F_k\}$ be an sequence of successive subsets of naturals such
 that $\lim_{k \to \infty} \# F_k = \infty$.
 \begin{itemize}
 \item[(i)] If $\{x_k\}_{k \inn}$ generates a spreading model equivalent to $c_0$, $F_k \in \mathcal{S}_1$ for $k \inn$ and $y_k = \sum_{i \in F_k} x_i$, then a subsequence of $\{y_k\}_{k \inn}$ generates an $\ell_1^\omega$ spreading model.
 \item[(ii)] Suppose $\{x_k\}_{k \inn}$ generates an $\ell_1^\omega$ spreading model, $F_k \in \mathcal{S}_\omega$ and $F_k$ is maximal $\mathcal{S}_\omega$ for each $k\inn$ (i.e. maximal in $\mathcal{S}_{\min F_k}$).  Let  $w_k=\sum_{j \in F_k} c_j x_i$ where $w_k$ is $(\min F_k, 3/\min F_k)$ s.c.c. Then a subsequence of $\{w_k\}_{k \inn}$ generates a $c_0$ spreading model.
 \end{itemize}
 \end{prp}

 \begin{proof}
The proof of (i) is identical to that of Proposition \ref{propmmm10}.

To prove (ii) it suffices to show $\alpha_{<\omega}(\{w_{k}\})=0$. We use Proposition \ref{higherorderindex}. Let $\e>0$. Find $j_0> 2/\e$. Let $j \geq j_0$ and
let $k_j \inn$ such that $36/\min F_{k_j} < \e$.  Let $k \geq k_j$, $\{\alpha_{q}\}_{q=1}^d$ be $\mathcal{S}_j$-admissible and very fast
growing sequence of $\alpha$-averages such that $s(\alpha_q) > j_0$ for $q = 1, \ldots , d$. Clearly, $j < F_{k_j}$. Using Lemma \ref{lemqqq6}
$$ \sum_{q=1}^d |\alpha_q(\sum_{j \in F_k} c^{F_k}_j x_j)| < \frac{1}{s(\alpha_1)} + 6\frac{3}{\min F_k}<\e.$$
 \end{proof}


\section*{Problems and Questions}
There are some questions and problems concerning the structure of
$\X$ and its dual which are open for us.\vskip3pt

\noindent {\bf Problem 1:} (i) Is $\X$ minimal?\vskip3pt

\noindent(ii) Does any sequence generating a $c_0$ spreading model
have a subsequence equivalent to some subsequence of the basis?

If this is true, then Proposition \ref{propmmm14} yields that $\X$
is sequentially minimal.

In particular, it is open to us whether two subsequences
$\{e_{i_m}\}_{m\inn}, \{e_{j_m}\}_{m\inn}$ of the basis, such that
$i_m < j_{m+1}$ and $j_m < i_{m+1}$ for all $m\inn$, are
equivalent.

Moreover, we do not know which class of Banach spaces in the
classification appearing in \cite{FR} the subspaces of $\X$
belong to.\vskip3pt

The next problem concerns the structure of
${\mathfrak{X}_{_{0,1}}^{n*}}$ and its strictly singular
operators.

\noindent {\bf Problem 2:} (i) Does any block sequence in
${\mathfrak{X}_{_{0,1}}^{n*}}$ contain a subsequence generating a
$c_0^k$, $k=1,\ldots,n$ or $\ell_1$ spreading model?\vskip3pt

\noindent (ii) Does any subspace of ${\mathfrak{X}_{_{0,1}}^{n*}}$
admit $c_0^k, k=1,\ldots,n$ and $\ell_1$ spreading models?

The latter is equivalent to the corresponding problem for
quotients of $\X$, namely if every quotient of $\X$ admits $c_0$
and $\ell_1^k$, $k=1,\ldots,n$ spreading models. Note that Cor.
\ref{cormmm16} yields that the same question for quotients of
${\mathfrak{X}_{_{0,1}}^{n*}}$ has an affirmative answer.\vskip3pt

\noindent (iii) Does ${\mathfrak{X}_{_{0,1}}^{n*}}$ satisfy that
whenever $S_1,\ldots,
S_{n+1}:{\mathfrak{X}_{_{0,1}}^{n*}}\rightarrow
{\mathfrak{X}_{_{0,1}}^{n*}}$ are strictly singular, then the
composition $S_1\cdots S_{n+1}$ is compact, as in $\X$?

A way of answering this affirmatively is to show that any subspace
$Y$ of $\X$, contains a further subspace which is complemented in
$\X$, which seems possible. \vskip6pt

As it was pointed out to us by Anna Pelczar-Barwacz, since $c_0$ and $\ell_1$ are both block finitely representable in every subspace of $\X$, it follows that $\X$ is arbitrarily distortable.





\begin{thebibliography}{99}

\bibitem{AA} D. E. Alspach and S. A. Argyros, {\em Complexity of weakly null
sequences}, Dissertationes Math. (Rozprawy Mat.) {\bf 321} (1992),
44.

\bibitem{ADST} G. Androulakis, P. Dodos, G. Sirotkin, G and V. G. Troitsky, {\em Classes of strictly singular operators and their
products}, Israel J. Math. {\bf 169} (2009), 221-250.

\bibitem{AD} S. A. Argyros and I. Deliyanni, {\em Banach spaces of the type of
Tsirelson} arXiv:math/9207206 (1992).

\bibitem{AD2} S. A. Argyros and I. Deliyanni, {\em Examples of asymptotic
$\ell_1$ Banach spaces}, Trans. Amer. Math. Soc. {\bf 349}, no. 3
(1997), 973-995.

\bibitem{ADT} S. A. Argyros, I. Deliyanni and A. G. Tolias, {\em Strictly
singular non-compact diagonal operators on HI spaces}, Houston J.
Math. {\bf 36}, no. 2 (2010), 513-566.

\bibitem{AGR} S. A. Argyros, G. Godefroy and H. P. Rosenthal, {\em Descriptive set theory and Banach
spaces}, Handbook of the Geometry of Banach Spaces, Vol. 2,
North-Holland,  Amsterdam (2003), pp. 1007-1069.

\bibitem{AH} S. A. Argyros and R. G. Haydon, {\em A hereditarily
indecomposable $\mathcal{L}_{\infty}$-space that solves the
scalar-plus-compact problem}, Acta Math. {\bf 206}, no. 1 (2011),
1-54.

\bibitem{AKT} S. A. Argyros, V. Kanellopoulos and K. Tyros,
{\em Finite Order Spreading Models} arXiv:1105.2732 (2011).

\bibitem{AM} S. A. Argyros and P. Motakis, {\em A reflexive HI space with the hereditary invariant subspace
property}     arXiv:1111.3603v2 (submitted).

\bibitem{AT} S. A. Argyros and A. Tolias, {\em Methods in the theory of hereditarily indecomposable Banach
spaces}, Memoirs of the American Mathematical Society {\bf 170}
(2004), vi+114.

\bibitem{AS} M. Aronszajn and K. T. Smith, {\em Invariant subspaces of
completely continuous operators}, Ann. of Math. {\bf 60} (1954),
345-350.

\bibitem{BS} A. Brunel and L. Sucheston,
{\em On B-convex Banach spaces}, Math. Systems Theory \textbf{7},
no.4 (1974), 294-299.

\bibitem{CS} P. G. Casazza and T. Shura, {\em Tsirelson's Space},
Spriger Lecture Notes 1363 (1989).

\bibitem{FR} V. Ferenczi and C. Rosendal, {\em Banach spaces without minimal
subspaces}, J. Funct. Anal. {\bf 257}, no. 1 (2009), 149-193.

\bibitem{FJ} T. Figiel and W. B. Johnson,
{\em A uniformly convex Banach space which contains no $l_p$},
Compositio Math. \textbf{29} (1974), 179-190.

\bibitem{H} N. D. Hooker, {\em Lomonosov's hyperinvariant subspace theorem for real spaces}, Math. Proc. Cambridge Philos. Soc. {\bf 89} (1981), 129-133.

\bibitem{J} R. C. James,
{\em Bases and refelexivity of Banach spaces}, Ann. of Math. (2)
\textbf{52} (1950), 518-527 .

\bibitem{L} V. I. Lomonosov, {\em Invariant subspaces of the family of operators
that commute with a completely continuous operator}, (Russian)
Funkcional. Anal. i Prilo\v zen. {\bf 7}, no. 3 (1973), 55-56.

\bibitem{M} V. D. Milman, {\em Operators of class $C_0$ and $C_0^*$}, Teor. Funkci\v \i\; Funkcional. Anal. i Priloz\v en. No 10 (1970), 15-26.

\bibitem{MT} V. D. Milman and N. Tomczak-Jaegermann, {\em Asymptotic $\ell_p$ spaces and
bounded distortions}, Banach Spaces, Contemp. Math. {\bf 144}
(1993), 173-196.

\bibitem{Nash1} C. St. J. A. Nash-Williams, {\em On better-quasi
ordering transfinite sequences}, Proc.  Cambridge Philos. Soc.
{\bf 64}, (1968), 273-290.

\bibitem{OS1} E. Odell and Th. Schlumprecht, {\em On the richness of the set of $p$'s
in Krivine's theorem}, Operator Theory, Advances and Applications
{\bf 77} (1995), 177-198.

\bibitem{OS2} E. Odell and Th. Schlumprecht, {\em A Banach space block finitely
universal for monotone bases}, Trans. Amer. Math. Soc. {\bf 352},
no. 4 (2000), 1859-1888.

\bibitem{R} C. J. Read, { \em Strictly singular operators and the invariant
subspace problem}, Studia Math. {\bf 132}, no. 3 (1999), 203-226.

\bibitem{S} G. Sirotkin, {\em A version of the Lomonosov invariant
subspace theorem for real Banach spaces} Indiana Univ. Math. J.
{\bf 54}, no. 1 (2005), 257-262.

\bibitem{Te} R. Teixeira, {\em On $\mathcal{S}_1$ strictly singular operatros} Ph.D thesis, University of Texas-Austin (2010).

\bibitem{T} B. S. Tsirelson, {\em Not every Banach space contains
$\ell_p$ or $c_0$}, Functional Anal. Appl. {\bf 8} (1974),
138-141.


\end{thebibliography}
\end{document}